\documentclass[a4paper,11pt]{amsart}

\usepackage{amsmath}
\usepackage{amsfonts}
\usepackage{amssymb}
\usepackage{mathrsfs}                  
\usepackage{amsthm}
\usepackage{amscd}
\usepackage{xcolor}
\usepackage[hyperfootnotes=false]{hyperref}
\usepackage{mathtools}
\usepackage{geometry}
\usepackage{bbm}
\usepackage{lipsum}

\newcommand\blfootnote[1]{%
  \begingroup
  \renewcommand\thefootnote{}\footnote{#1}%
  \addtocounter{footnote}{-1}%
  \endgroup
}

\numberwithin{equation}{section}
\theoremstyle{plain}
\newtheorem{theorem}{Theorem}[section]
\newtheorem{proposition}[theorem]{Proposition}
\newtheorem{lemma}[theorem]{Lemma}
\newtheorem{corollary}[theorem]{Corollary}
\newtheorem{remark}[theorem]{Remark}
\newtheorem{definition}[theorem]{Definition}
\newtheorem{example}[theorem]{Example}
\newtheorem{assumption}[theorem]{Assumption}

\newcommand{\dd}{\,\mathrm{d}}

\newcommand{\E}{\mathbb{E}}
\newcommand{\R}{\mathbb{R}}
\newcommand\bplim{\operatorname{bp-lim}}
\newcommand{\thor}{t_0}

\renewcommand{\d}{\mathrm{d}}
\renewcommand{\P}{\mathbb{P}}

\title[Fokker--Planck equations for time-dependent operators]%
{Existence and uniqueness results for time-inhomogeneous time-change equations and Fokker--Planck equations
}

\author[D\"oring]{Leif D\"oring}
\address{Leif D\"oring, University of Mannheim, Germany}
\email{doering@uni-mannheim.de}

\author[Gonon]{Lukas Gonon}
\address{Lukas Gonon, University of St. Gallen, Switzerland}
\email{lukas.gonon@unisg.ch}

\author[Pr\"omel]{David J. Pr\"omel}
\address{David J. Pr\"omel, University of Oxford, United Kingdom}
\email{proemel@maths.ox.ac.uk}

\author[Reichmann]{Oleg Reichmann}
\address{Oleg Reichmann, European Investment Bank, Luxembourg}
\email{o.reichmann@eib.org}

\date{\today}

\begin{document}

\begin{abstract}
  We prove existence and uniqueness of solutions to Fokker--Planck equations associated to Markov operators multiplicatively perturbed by degenerate time-inhomogeneous coefficients. Precise conditions on the time-inhomogeneous coefficients are given. In particular, we do not necessarily require the coefficients to be neither globally bounded nor bounded away from zero. The approach is based on constructing random time-changes and studying related martingale problems for Markov processes with values in locally compact, complete and separable metric spaces.  
\end{abstract}

\maketitle
\frenchspacing\blfootnote{The views expressed in this article are those of the authors and not necessarily of the European Investment Bank.}



\section{Introduction}

The study of existence and uniqueness of solutions to Fokker--Planck equations, also known as forward Kolmogorov equations, is a classical topic of great current interest. One reason are the numerous applications which arose over the past decades such as in the theory of stochastic processes and of (partial) differential equations. 

In this article, we consider the Fokker--Planck equation associated to a linear operator~$A$ which is assumed to be the (infinitesimal) generator of a Markov process with values in a metric space~$E$ (for instance $E = \R^d$ and~$A$ is an integro-differential operator), and to a degenerate \textit{time-inhomogeneous} coefficient~$\sigma$. More precisely, we establish sufficient conditions such that there exists a unique family of probability measures~$(p(t,\cdot))_{t\in [0,\thor]}$ on~$E$ solving the Fokker--Planck equation 
\begin{equation}\label{eq:intro}
  \int_E f(x)\, p(t,\d x) - \int_E f(x) \,\mu_0(\d x) = \int_0^t \int_E \sigma(s,x) \mathcal{A}f(x)\, p(s,\d x) \dd s,\quad t \in [0,\thor],
\end{equation}
for all ``sufficiently nice'' test functions~$f$ and given the initial condition $\mu_0$ and $\thor>0$. We refer to Theorem~\ref{thm:KolmogorovUniqueness} for the exact formulation of our existence and uniqueness result. Let us point out that the coefficient~$\sigma$ is not only time-dependent but also~$\sigma$ is neither (necessarily) globally bounded nor bounded away from zero.

For a \textit{time-homogeneous} coefficient~$\sigma$, the existence and uniqueness result provided in Theorem~\ref{thm:KolmogorovUniqueness} is well-known, see, e.g., the book~\cite{Ethier1986a} or~\cite{Kurtz1998}. Furthermore, classical results on multiplicative perturbations of Feller generators and time-changed L\'evy processes allow for weak regularity assumptions on~$\sigma$, see e.g. \cite[Thm.~4.1]{Boettcher2013} and the original reference~\cite{Lumer1973}, \cite{Engelbert1985} and the references therein. However, note that these results do not deal with the question whether the Fokker--Planck equation uniquely determines the law of the time-changed process.

For a \textit{time-inhomogeneous} coefficient~$\sigma$, existence and uniqueness has been studied in various situations, for instance, for $\sigma$ bounded away from~$0$, we refer to \cite{Stroock1975}, \cite{Bass1988}, \cite{Bentata2009} and further references therein and, for globally bounded $\sigma$ and $\mathcal{A}$ being the generator of a diffusion, see~\cite{Figalli2008}. However, to the best of our knowledge, the previous literature does not cover our assumptions on the operator~$\mathcal{A}$ and~$\sigma$.
 
An application, where such general conditions on the coefficient~$\sigma$ are essential, can be found in the accompanying paper~\cite{Doering2017} studying the solvability of the Skorokhod embedding problem (SEP) for L\'evy processes~$L$. Let us recall that the version of the SEP dealt with in~\cite{Doering2017} asks to find an integrable stopping time~$\tau$ such that~$L_\tau$ is distributed according to a given law~$\mu_1$. The key idea of the approach developed in~\cite{Doering2017} is to find a construction of a coefficient~$\sigma$ such that there exists a unique solution~$(p(t,\cdot))_{t\in [0,\thor]}$ to the Fokker--Planck equation~\eqref{eq:intro} with $p(\thor,\cdot)=\mu_1$, where $\mathcal{A}$ is assumed to be the generator of the given L\'evy process. Needless to say, for this construction to work, weak regularity conditions on~$\sigma$ are crucially required. Moreover, the general setting considered in the present article paves the way to apply the approach of~\cite{Doering2017} to solve the Skorokhod embedding problem for other stochastic processes as well, such as Markov chains or multi-dimensional Brownian motion.

Apart from the indisputably justified interest of investigating Fokker--Planck equations on its own, further motivation to prove existence and uniqueness results for Fokker--Planck equations stems from its long list of recent applications. Just to name a few, let us mention the explicit construction of peacocks (see~\cite[Chapter~6]{Hirsch2011}), the construction of (martingale) diffusions matching prescribed marginal distributions at given (random) times (see e.g. \cite{Cox2011} and \cite{Ekstrom2013}) or Dupire's formula in mathematical finance (see~\cite{Dupire1994} and \cite{Carr2004}). Further potential applications include probabilistic representations of the solution to irregular porous media type equations~\cite{Barbu2011} and measure-valued martingales \cite{Veraguas2017}. The results obtained in the present article will be a useful tool when extending any of these applications to more general Markov processes such as, e.g., L\'evy type processes and the associated generators. 

\medskip
\noindent\textbf{Acknowledgement:} L.G. and D.J.P acknowledge generous support from ETH Z\"urich, where a major part of this work was completed.

\medskip
\noindent\textbf{Organization of the paper:} In Section~\ref{sec:notation} the notation, definitions and assumptions are introduced. Section~\ref{sec:time-change} is devoted to the construction of a solution to the martingale problem and to the Fokker--Planck equation. The existence and uniqueness result (Theorem~\ref{thm:KolmogorovUniqueness}) for Fokker--Planck equations with degenerate time-inhomogeneous coefficients is proven in Section~\ref{sec:uniqueness}. Auxiliary results for the presented construction of random time-changes are provided in Section~\ref{sec:appendix}.

\section{Notation, Definitions and Assumptions}\label{sec:notation} 

Our definitions and notation are precisely as in \cite{Ethier1986a} with two exceptions pointed out in Remark~\ref{rmk:expections}.
  
Throughout the whole article, the underlying stochastic basis consists of a probability space $(\Omega, \mathcal{F}, \mathbb{P})$ and a filtration $(\mathcal{F}_t)_{t\geq 0}$ satisfying the usual conditions of completeness and right-continuity. Moreover, we fix a locally compact, complete, separable metric space $(E,d)$ with metric $d$ and denote by $\mathcal{B}(E)$ its Borel $\sigma$-algebra. For $x \in E$ and $\varepsilon > 0$, set $B_\varepsilon(x):= \{y \in E \, :\, d(x,y) < \varepsilon\}$. $\mathcal{P}(E)$ denotes the set of probability measures on $(E,\mathcal{B}(E))$.

The space of continuous functions $f\colon \R \to \R$ satisfying $\lim_{|x|\to \infty } f(x) = 0$ is denoted by $C_0(\mathbb{R})$. For $n\in \mathbb{N}$ let $C^n_0(\R)$ be the subset of functions $f\in C_0(\R)$ such that $f$ is $n$-times differentiable and all derivatives of order less or equal to $n$ belong to $C_0(\R)$ and we set $C_0^\infty(\R):=\bigcap_{n\in \mathbb{N}}C^n_0(\R)$. The spaces of functions with compact support $C_c(\R)$, $C_c^n(\R)$ and $C_c^\infty(\R)$ are defined analogously. 

The space $D_E[0,\infty)$ stands for all maps $\omega \colon [0,\infty) \to E$ which are right-continuous and have a left-limit at each point $t\in [0,\infty)$ (short: RCLL paths). The space $D_E[0,\infty)$ is equipped with the Skorokhod topology $(J1)$, see \cite[Chap.~3, Sec. 5]{Ethier1986a} for the precise definition. 
$B(E)$ denotes the space of real-valued, bounded, measurable functions on $E$ and $\|\cdot\|$ is the sup-norm. $C_0(E)$ denotes the space of continuous functions that vanish at infinity and $C_b(E)$ the space of bounded continuous functions on $E$. A sequence $\{f_n\}_{n \in \mathbb{N}} \subset B(E)$ converges \textit{boundedly and pointwise} to $f \in B(E)$ (denoted by $\bplim_{n\to \infty} f_n = f$) if  $\sup_n \|f_n \| < \infty$ and $\lim_{n \to \infty} f_n(x) = f(x)$ for all $x \in E$. $U \subset B(E)$ is called \textit{bp-closed}, if $\{f_n\}_{n \in \mathbb{N}} \subset U$ and $\bplim_{n \to \infty} f_n = f$ implies $f \in U$. For  $V \subset B(E)$, we define $\text{bp-closure}(V)$ as the smallest subset of $B(E)$ which is bp-closed and contains $V$. A sequence $\{(f_n,g_n)\}_{n \in \mathbb{N}} \subset B(E)\times B(E) $ is said to converge to $(f,g) \in B(E) \times B(E)$ boundedly and pointwise (denoted by $\bplim_{n\to \infty} (f_n,g_n) = (f,g)$) if $\bplim_{n\to \infty} f_n = f$ and $\bplim_{n\to \infty} g_n = g$. The definitions of bp-closed and bp-closure are then defined analogously for subsets of $B(E) \times B(E)$.\smallskip

Following the martingale problem approach in \cite{Ethier1986a}, consider a given $D(\mathcal{L}) \subset C_b(E)$ and $\mathcal{L}\colon D(\mathcal{L}) \to C_b(E)$ linear. $(D(\mathcal{L}),\mathcal{L})$ is said to be \textit{conservative} if 
\begin{equation}\label{eq:defConservative} 
  \text{ there exists } \{h_n\}_{n \in \mathbb{N}} \subset D(\mathcal{L})  \text{ such that } (1,0) = \bplim_{n \to \infty} (h_n, \mathcal{L} h_n).
\end{equation} 
For a stochastic process $(X_t)_{t \geq 0}$ we set $\mathcal{F}_t^X := \sigma(X_s \, :\, s \leq t)$. A \textit{solution to the martingale problem for $(\mathcal{L},\mu)$} is a progressively measurable $E$-valued stochastic process $(X_t)_{t \geq 0}$ defined on some probability space $(\tilde{\Omega},\tilde{\mathcal{F}},\tilde{\mathbb{P}})$ such that for each $h \in D(\mathcal{L})$, the process 
\begin{equation}\label{eq:martingaleProblemDef} 
  h(X_t) - h(X_0) - \int_0^t \mathcal{L} h (X_s) \dd s,  \quad t \geq 0,
\end{equation}
is an $(\mathcal{F}^X_t)_{t\geq 0}$-martingale and $\tilde{\mathbb{P}} \circ X_0^{-1} = \mu$. \textit{Uniqueness} is said to hold for the martingale problem for $(\mathcal{L},\mu)$ if any two solutions $X$, $\tilde{X}$ have the same finite-dimensional distributions. The martingale problem for $(\mathcal{L},\mu)$  is said to be \textit{well-posed} if there exists a solution and uniqueness holds. A solution to the $D_E[0,\infty)$-martingale problem (or RCLL-martingale problem) for $(\mathcal{L},\mu)$ is an RCLL process that is a solution to the martingale problem for $(\mathcal{L},\mu)$. \textit{Uniqueness} is said to hold for the RCLL-martingale problem for $(\mathcal{L},\mu)$ if any two solutions to the RCLL-martingale problem for $(\mathcal{L},\mu)$ have the same law on $D_E[0,\infty)$. The RCLL-martingale problem for $(\mathcal{L},\mu)$ is said to be \textit{well-posed} if there exists a solution and uniqueness holds.

Similarly, for a linear operator $(D(\mathcal{A}),\mathcal{A})$ and a measurable function $\sigma\colon [0,\infty) \times E \to [0,\infty)$, a solution to the (time-inhomogeneous) martingale problem for $(\sigma \mathcal{A}, \mu)$ is a progressively measurable $E$-valued process $X$ defined on some probability space $(\tilde{\Omega},\tilde{\mathcal{F}},\tilde{\mathbb{P}})$ such that for each $f \in D(\mathcal{A})$ the process 
\begin{equation}\label{eq:InHomMartingaleProblemDef} 
  f(X_t) - f(X_0) - \int_0^t \sigma(s,X_s)\mathcal{A} f (X_s) \dd s,  \quad t \geq 0,
\end{equation}
is an $(\mathcal{F}^X_t)_{t\geq 0}$-martingale and $\tilde{\mathbb{P}} \circ X_0^{-1} = \mu$. Uniqueness, well-posedness and the corresponding concepts among $D_E[0,\infty)$-processes are defined analogously to the time-homogeneous case. For $\mu \in \mathcal{P}([0,\infty))$ and $\nu \in \mathcal{P}(E)$ we write $\mu \otimes \nu$ for the product measure generated by $\mu$ and $\nu$ on $[0,\infty)\times E$. If~$F$ is a measurable space, $\delta_x$ denotes the Dirac measure at $x \in F$. For $V \subset B(E)$, $\mathrm{span}(V)$ denotes the smallest linear subspace of $B(E)$ containing~$V$, i.e. the set of all finite linear combinations of elements of~$V$. 

\begin{remark}\label{rmk:expections}
  The notation differs from the book~\cite{Ethier1986a} in two following respects:
  
  Firstly, in \cite{Ethier1986a} $(D(\mathcal{L}),\mathcal{L})$ is said to be conservative if
  \begin{equation}\label{eq:EKconservative} 
    (1,0) \in \text{bp-closure}(\{(h,\mathcal{L} h) \, :\, h \in D(\mathcal{L}) \}).
  \end{equation}
  While our requirement~\eqref{eq:defConservative} implies~\eqref{eq:EKconservative}, the converse is not true in general, cf. \cite[Chap.~3, Sec.~4]{Ethier1986a}.

  Secondly, (not necessarily RCLL) solutions to martingale problems are required to be progressively measurable in our context, as in~\cite{Kurtz1998}.
\end{remark}

\begin{remark}\label{rmk:martingaleProblemBpClosure}
  The motivation for the definition of bp-closure is as follows: Suppose $(D(\mathcal{L}),\mathcal{L})$ is a linear operator on $C_b(E)$ and $X$ is a solution to the martingale problem for $(\mathcal{L},\mu)$ for some $\mu \in \mathcal{P}(E)$. Then by the dominated convergence theorem for conditional expectations, the set of functions $(h,g) \in B(E) \times B(E)$, for which 
  \begin{equation}\label{eq:martingaleProblemBpClosure} 
    h(X_t) - h(X_0) - \int_0^t g (X_s) \dd s,  \quad t \geq 0, 
  \end{equation}
  is an $(\mathcal{F}^X_t)_{t\geq 0}$-martingale, is bp-closed and so \eqref{eq:martingaleProblemBpClosure} is a martingale for all $(h,g) \in \text{bp-closure}(\{(h,\mathcal{L} h) \, :\, h \in D(\mathcal{L}) \})$.
\end{remark}

Finally, we provide two definitions analogous to the conditions imposed in \cite[Thm.~35.4~(iii)]{Sato1999} and \cite[Chap.~6, Thm.~1.1]{Ethier1986a}. The first one is a (strong) recurrence property, the second is essential for studying uniqueness of time-change equations~\eqref{eq:ode time-change}. Denote by~$Z$ the coordinate process on $D_E[0,\infty)$.

\begin{definition}\label{def:recurrence} 
  A probability measure $P$ on $(D_E[0,\infty),\mathcal{B}(D_E[0,\infty))$ is called recurrent if
  \begin{equation}\label{eq:recurrent} 
    P\left( \int_0^\infty \mathbbm{1}_{B_a(Z_0)}(Z_t) \dd t = \infty \right) = 1 \quad \text{ for every } a > 0. 
  \end{equation}
\end{definition}

\begin{definition}\label{def:regular} 
  Let $H \colon E \to [0,\infty)$ be measurable and $P \in \mathcal{P}(D_E[0,\infty))$. $H$ is called regular for~$P$ if $P$-a.s.
  \begin{equation*}
    \inf \left\lbrace s \in [0,\infty)\,:\,\int_0^s H(Z_u)^{-1} \dd u =\infty \right\rbrace = \rho \quad \text{ and } \quad H(Z_\rho)=0 \text{ on } \{\rho < \infty\}
  \end{equation*} 
  where  
  \begin{equation*}
    \rho := \inf \left\lbrace s \in [0,\infty)\,:\, H(Z_s) = 0 \right\rbrace.
  \end{equation*}
\end{definition}

\begin{example}\label{ex:regularAt0isregular} 
  Let $E=\R$ and $Z$ be a $P$-Brownian motion. Suppose $H\colon \mathbb{R} \to [0,\infty)$ satisfies
  \begin{equation}\label{eq:HregularAt0} 
    \{ x \in \mathbb{R}\, :\, H(x) = 0 \} = I(H),
  \end{equation} 
  where $I(H)$ is the closed set
  \begin{equation}\label{eq:defintion I}
    I(H) := \left\lbrace x \in \mathbb{R}\,:\, \forall \varepsilon > 0 : \int_{x-\varepsilon}^{x+\varepsilon} \frac{\dd y}{H(y)} = \infty \right\rbrace. 
  \end{equation}
  Then $H$ is regular for $P$. Indeed, since $I(H)$ in \eqref{eq:defintion I} is closed and \eqref{eq:HregularAt0} holds, $H^{-1}(\{0\})$ is closed. Hence, for $P$-a.e. $\omega$ with $\rho(\omega)< \infty$,  $H(Z_{\rho(\omega)}(\omega)) = 0$ by (right-)continuity. Hence, the second part of the definition is established and the first part follows directly from \eqref{eq:HregularAt0} and \cite[Chap.~5, Lem.~5.2]{Karatzas1991a}.
\end{example}

\subsection{Assumptions}

The following assumptions are used at different places throughout the article. Our set of assumptions is split in such a way that we can distinguish as good as possible between assumptions on the stochastic process and on the time-inhomogeneous coefficient~$\sigma$.

\begin{assumption}[Regularity of $\sigma$]\label{ass:sigma}
  Let $\thor > 0$ and $\sigma \colon [0,\infty) \times E \to [0,\infty)$ be given. Suppose $\sigma$ is of the form $\sigma (t,x):= H(x) \tilde{\sigma}(t,x)$ for $(t,x)\in[0,\infty) \times E$ with $\tilde{\sigma}(t,x) \equiv 0$ for $t> \thor$ and such that
  \begin{enumerate}
    \item[(i)] $H \colon E \to [0,\infty)$ is measurable,
    \item[(ii)] $\tilde{\sigma} \colon [0,\thor] \times E \to (0,\infty)$ is measurable and satisfies the following: for each compact set $K \subset E$ and $S \in (0,\thor)$ there exists $C_1,C_2,C_3 >0$ such that 
          \begin{equation*}
            |\tilde{\sigma}(t,x) - \tilde{\sigma}(s,x)|  \leq C_1 |t - s|\quad \text{and}\quad C_2 \leq \tilde{\sigma}(t,x) \leq  C_3,
          \end{equation*}
          for all $s,t \in [0,S]$ and for all $x \in K$, where $C_3$ does not depend on $S$ (but it may depend on $K$).
  \end{enumerate}
\end{assumption}

\begin{assumption}\label{ass:FellerProcess}
  Let $\mathcal{D} \subset C_0(E)$ and $\mathcal{A}\colon\mathcal{D} \to C_0(E)$ be linear so that
  \begin{enumerate}
    \item[(i)] $(\mathcal{D},\mathcal{A})$ is conservative, $\mathcal{D}$ is dense in $C_0(E)$ and an algebra\footnote{That means $\mathcal{D}\subset C_0(E)$ is an algebra with respect to the addition and multiplication induced by $C_0(E)$.} in $C_0(E)$,
    \item[(ii)] for any $\mu_0 \in \mathcal{P}(E)$, the RCLL-martingale problem for $(\mathcal{A},\mu_0)$ is well-posed.
  \end{enumerate} 
\end{assumption}

\begin{remark}\label{rmk:Px} The most important special case of Assumption \ref{ass:FellerProcess}~(ii) is $\mu_0 = \delta_x$. In this case the corresponding law on $D_E[0,\infty)$ of the RCLL-solution to the martingale problem is denoted by $P_x$.\end{remark}  

\begin{example}\label{rmk:FellerEx} 
  Let $\bar{\mathcal{A}}$ be the generator of a Feller semigroup on $C_0(E)$ with domain $D(\mathcal{\bar{A}})$ and $\mathcal{D}$ be a core for $\bar{\mathcal{A}}$ (see \cite{Ethier1986a}). Suppose $\mathcal{D}$ is an algebra in $C_0(E)$ and denote by $\mathcal{A}$ the restriction of $\bar{\mathcal{A}}$ to $\mathcal{D}$. Then $\mathcal{D}$ is dense, $(\mathcal{D},\mathcal{A})$ satisfies the assumptions of \cite[Chap.~4, Thm.~2.2]{Ethier1986a} and thus, by \cite[Chap.~4, Cor.~2.8]{Ethier1986a}, $(\mathcal{D},\mathcal{A})$ is conservative in the sense of \eqref{eq:EKconservative} (from the proof of \cite[Chap.~4, Cor.~2.8]{Ethier1986a} also conservative in our sense \eqref{eq:defConservative}). Furthermore, by \cite[Chap.~4, Thm.~2.7 and Thm.~4.1]{Ethier1986a}, for any $\mu_0 \in \mathcal{P}(E)$ the martingale problem for $(\mathcal{A},\mu_0)$ is well-posed and the solution has sample paths in $D_E[0,\infty)$. Hence, Assumption~\ref{ass:FellerProcess} indeed holds. 
\end{example}

\begin{assumption}[Regularity of $H$]\label{ass:Hregular}
  Let $(P_x)_{x \in E}$ as in Assumption~\ref{ass:FellerProcess} and $H \colon E \to [0,\infty)$ measurable. Assume that for any $x \in E$, $H$ is regular for~$P_x$ (in the sense of Definition~\ref{def:regular}).
\end{assumption}

Assumption~\ref{ass:Hregular} is needed to guarantee uniqueness of the time-change equations~\eqref{eq:ode time-change}. Proposition~\ref{prop:regular} provides a useful criterion to verify it.

\begin{assumption}[Recurrence and boundedness]\label{ass:Recurrence} 
  Let $(\mathcal{D},\mathcal{A})$ and $(P_x)_{x \in E}$ be as in Assumption~\ref{ass:FellerProcess} and Remark~\ref{rmk:Px} and let $\thor > 0$ and $\sigma = H \tilde{\sigma}$ be as in Assumption~\ref{ass:sigma}. Assume that
  \begin{enumerate}
    \item[(i)] for any $x \in E$, $P_x$ is recurrent (in the sense of Definition~\ref{def:recurrence}) and $H$ is bounded on compacts,
    \item[(ii)] $\sigma \mathcal{A} f \in C_0([0,\thor] \times E)$ for all $f \in \mathcal{D}$.
  \end{enumerate}
\end{assumption}

Assumption~\ref{ass:Recurrence}~(ii) can be seen as a weak locality assumption, which is always satisfied for Brownian motion:  

\begin{example}
  Set $\mathcal{D} := C_c^\infty(\mathbb{R})$ and $\mathcal{A} f(x) := \frac{1}{2} f''(x)$ for $f \in \mathcal{D}$. By \cite[Chap.~5, Prop.~1.1]{Ethier1986a} and Example~\ref{rmk:FellerEx}, Assumption~\ref{ass:FellerProcess} is satisfied and under $P_x$ the canonical process~$Z$ is a Brownian motion started from $x$ so (the first part of) Assumption~\ref{ass:Recurrence}~(i) holds. A sufficient condition for Assumption~\ref{ass:Recurrence}~(ii) to hold is that $\sigma$ is continuous on $[0,\thor]\times \mathbb{R}$: If this is true, then $(t,x) \mapsto \sigma(t,x) \mathcal{A}f(x)$ is continuous and even compactly supported for all $f \in \mathcal{D}$.
\end{example}

\begin{assumption}[Boundedness of $\sigma$]\label{ass:SigmaBounded} 
  $\sigma \colon [0,\infty) \times E \to [0,\infty)$ is bounded.
\end{assumption}

As we will see in Section~\ref{sec:time-change}, Assumption~\ref{ass:Recurrence}~(i) or \ref{ass:SigmaBounded} ensure that there exists a non-exploding solution to the time-change equation~\eqref{eq:ode time-change} below.

\section{Time-Inhomogeneous Time-Changes for Markov Processes}\label{sec:time-change}

Given a Markov process~$M$ with generator~$\mathcal{A}$ and a sufficiently regular time-inhomogeneous coefficient $\sigma$, our aim is to obtain a Markov process~$X$ with generator~$\sigma \mathcal{A}$. The new Markov process~$X$ is identified as a time-change of $M$, where the time-change~$\tau$ is characterized by the pathwise Carath\'eodory differential equation 
\begin{equation}\label{eq:ode time-change}
  \tau(t) = \int_0^t \sigma(s,M_{\tau(s)}) \dd s, \quad t\in [0,\thor].
\end{equation}

\subsection{Constructing the Time-Change}

Due to the time-inhomogeneity of the coefficient $\sigma$ in the differential equation~\eqref{eq:ode time-change}, we need to include the time variable $t$ in the state space of the time-changed process $X$. This time-inhomogeneity prevents us to directly rely on well-known results as for example \cite[Chap.~6, Thm.~1.1]{Ethier1986a}. Therefore, we verify as a first step that equation~\eqref{eq:ode time-change} indeed has a solution. 

\begin{lemma}\label{lem:timechange}
  Let $M$ be an $E$-valued process on $(\Omega,\mathcal{F},\P)$ with $\P$-almost surely RCLL sample paths. Denote by $P$ the law on $D_E[0,\infty)$ of $M$. Assume that 
  \begin{itemize}
    \item $\thor > 0$ and $\sigma = H \tilde{\sigma}$ are given as in Assumption~\ref{ass:sigma},
    \item $H$ is regular for $P$ (in the sense of Definition~\ref{def:regular}),
    \item either $H$ is bounded on compacts and $P$ is recurrent (see Definition~\ref{def:recurrence}) or Assumption~\ref{ass:SigmaBounded} holds.
  \end{itemize}
  Then there exists a family of random times $(\tau(t))_{t \in [0,\thor]}$ such that 
  \begin{enumerate}
    \item[(i)] $\tau \colon [0,\thor] \to [0,\infty)$ is non-decreasing and absolutely continuous, $\P$-a.s.,  
    \item[(ii)] $\tau (\thor)$ is finite, $\P$-a.s., 
    \item[(iii)] $\tau$ solves the Carath\'eodory differential equation~\eqref{eq:ode time-change} for $M$.
  \end{enumerate}
\end{lemma}

\begin{proof}
  Using the conventions $\inf \emptyset := \infty$ and $[0,0):=\{0\}$, we define the random time
  \begin{equation}\label{eq:rho}
    \rho := \inf \left\lbrace s \in [0,\infty)\,:\,\int_0^s H(M_u)^{-1} \dd u =\infty \right\rbrace 
  \end{equation}
  and the random times 
  \begin{equation}\label{eq:stoppingTime} 
    \tau(t) := \begin{cases}
                   \inf\{s \in [0,\rho)\,:\, \mathcal{T}(s) \geq t\} \wedge \rho  & \text{if }t\in [0,\thor)\\
                   \sup_{s \in [0,\thor)} \inf\{r \in [0,\rho)\,:\, \mathcal{T}(r) \geq s\} \wedge \rho  & \text{if }t=\thor
                 \end{cases}      
  \end{equation} 
  where $\mathcal{T}$ is given by the Carath\'eodory differential equation 
  \begin{equation}\label{eq:Delta} 
    \mathcal{T}(s) = \int_0^s \sigma(\mathcal{T}(r),M_r)^{-1} \dd r,\quad s \in [0, \tau(\thor)).
  \end{equation}
  In other words, $\tau$ is the right inverse of $\mathcal{T}$ until $\sigma$ becomes $0$, then one sets $\tau=\rho$. To prove that $(\tau(t))_{t\in [0,\thor]}$ is well-defined, it is sufficient to show that equation~\eqref{eq:Delta} has indeed a unique solution on the interval $[0,\tau(t)\wedge T]$ for every $t\in [0,\thor)$ and $T\in [0,\rho)$. Let us fix $t \in (0,\thor)$, an RCLL sample path of $M$ denoted by $(M_s(\omega))_{s\in [0,T]}$ for $\omega \in \Omega$ and $T\in[0,\rho(\omega))$. By the RCLL property, the path $(M_s(\omega))_{s\in [0,T]}$ is contained in a compact set $K \subset E$, i.e. $\{M_{s}(\omega) \,:\,s \in [0,T]\} \subset K$, and $N(\omega):=\{s \in [0,T]\,:\, H(M_s(\omega)) = 0 \}$ is a Lebesgue null set by definition of $\rho$. Therefore, $\sigma(r,M_s(\omega)) > 0$ and so $\gamma(r,s) := \sigma(r,M_s(\omega))^{-1}$ is well-defined for $r \in [0,t]$ and $s\in[0,T]\setminus N(\omega)$ and we set $\gamma(\cdot,s):= 1$ for $s\in N(\omega)$. Since $\{M_{s}(\omega)\,:\, s \in [0,T]\} \subset K$, by Assumption~\ref{ass:sigma} on $\sigma=H\tilde \sigma$ there exist $C_1, C_2 > 0$ such that $C_2 H(M_t)^{-1} \leq \gamma(u,s) $ and 
  \begin{equation*}
    |\gamma(u,s) - \gamma(v,s)| = H(M_s(\omega))^{-1}\bigg|\frac{\tilde{\sigma}(u,M_s(\omega))-\tilde{\sigma}(v,M_s(\omega))}{ \tilde{\sigma}(u,M_s(\omega))\tilde{\sigma}(v,M_s(\omega))}\bigg| \leq C_2^2 C_1 H(M_s(\omega))^{-1} |u-v| 
  \end{equation*}
  for all $u,v \in [0,t]$ and $s \in [0,T]$. Thus, $\gamma$ satisfies the assumptions of Lemma~\ref{lem:timeinverse}, which says that there exists a unique solution $\mathcal{T}$ of the Carath\'eodory differential equation~\eqref{eq:Delta} on the interval $[0,\tau(t)\wedge T]$. Moreover, since now $\tau(s)$ is well-defined for all $s\in [0,\thor)$ and $\tau(\thor)= \sup_{s < \thor} \tau(s)$, $\tau(\thor)$ is also well-defined.\smallskip 
  
  Note that if Assumption~\ref{ass:sigma}~(ii) holds also for $S=t_0$, we set
  \begin{equation*} 
    \tau(t) :=  \inf\{s \in [0,\rho)\,:\, \mathcal{T}(s) \geq t\} \wedge \rho, \quad t\in [0,\thor],
  \end{equation*} 
  instead of~\eqref{eq:stoppingTime} and the above argument works for $t=t_0$ as well.\smallskip
  
  \textit{(i)} By definition of $\mathcal{T}$ through equation~\eqref{eq:Delta}, $\mathcal{T}$ is absolutely continuous and strictly increasing on $[0,\tau(t))$ for every $t\in [0,\thor)$ and thus invertible with $\tau(s) = \mathcal{T}^{-1}(s)$ for $s \in [0,t)$. This implies that $\tau$ is also non-decreasing and absolutely continuous on $[0,\thor]$ (cf. \cite[Thm.~1.7 and Ex.~3.21]{Leoni2009}).\smallskip 

  \textit{(ii)} To verify that $\tau(\thor) < \infty$, $\P$-a.s., suppose first $H$ is bounded on compacts and $P$ is recurrent. Let $N:=\{ \omega \in \Omega \,:\, \tau(\thor)(\omega) = \infty\}$. Then for $\omega \in N$, equation~\eqref{eq:Delta} has a solution on $[0,\infty)$ and $\mathcal{T}(t)(\omega) < \thor$ for all $t \geq 0$.
  Fixing some $a > 0$, we notice that by Assumption~\ref{ass:sigma} and since $H$ is bounded on compacts, there exists a constant $C_1 > 0$ such that $\sigma(\mathcal{T}(s),M_s)^{-1} \geq C_1$ for $s\in \{t\geq 0\,:\, M_t(\omega) \in B_a(M_0(\omega))\}$ and $\omega \in N$. We therefore have
  \begin{align}\label{eq:DeltaLim}
    t_0 \geq \lim_{t\to \infty} \mathcal{T}(t) = \int_0^\infty \sigma(\mathcal{T}(s),M_s)^{-1} \dd s &\geq \int_0^\infty \mathbbm{1}_{B_a(M_0)}(M_s) \sigma(\mathcal{T}(s),M_s)^{-1} \dd s\nonumber \\
    &\geq C_1 \int_0^\infty \mathbbm{1}_{B_a(M_0)}(M_s) \dd s
  \end{align}
  on $N$. However, the right-hand side of \eqref{eq:DeltaLim} is infinite, $\P$-a.s., by the recurrence assumption and hence \eqref{eq:DeltaLim} can only hold on a null set. Thus $N$ is a $\P$-null set, i.e. $\tau(\thor) < \infty$, $\P$-a.s., as claimed.
  Supposing Assumption~\ref{ass:SigmaBounded} holds, a similar argument works.
  \smallskip 
  
  \textit{(iii)} For every $t\in [0,\thor]$ such that $\tau(t)<\rho$, one observes that 
  \begin{equation}\label{eq:proof lemma timechange 2}
    1=\frac{\d}{\d s}(\mathcal{T} (\tau(s))= \frac{\d}{\d s}\mathcal{T}(\tau (s) )\frac{\d}{\d s}\tau (s)= \sigma (\tau (s),M_{\tau (s)})^{-1} \frac{\d}{\d s}\tau (s)
  \end{equation}
  for almost all $s\in [0,\mathcal{T}(\tau(t))]$, where the chain rule (c.f. \cite[Thm.~3.44]{Leoni2009}) and \eqref{eq:Delta} was used. Therefore, the absolutely continuity of $\tau$ and identity \eqref{eq:proof lemma timechange 2} show that $\tau$ indeed solves the desired integral equation \eqref{eq:ode time-change} since
  \begin{equation*}
    \tau(t)= \int_0^t \frac{\d}{\d s}\tau(s) \dd s = \int_0^t \sigma(s,M_{\tau(s)}) \dd s.
  \end{equation*} 
  For every $t\in [0,\thor]$ such that $\tau(t)\geq \rho$, we denote $\chi := \inf \{t\in [0,\thor]\,:\,\tau(t)\geq \rho\}$. Notice that $\tau$ is constant (equal to $\rho$) on $[\chi,\thor]$ and, $\P$-a.s., $H(M_\rho) = 0$, by the assumption that $H$ is regular for the law of $M$.\footnote{\label{fn:regular} In fact, here only \eqref{eq:regularInLemma1} is used and so the statement of the Lemma~\ref{lem:timechange} could be modified accordingly. \eqref{eq:regularInLemma2} is only used for uniqueness of the time-change in Lemma~\ref{lem:uniquenessOfTimeChange} below.}
  In particular, $\tau(t)$ satisfies equation~\eqref{eq:ode time-change} for every $t\in [\chi,\thor]$ as well. Recall that the assumption on $H$ means that $\rho$ defined in~\eqref{eq:rho} satisfies 
  \begin{align}\label{eq:regularInLemma1}
    & H(M_\rho)=0 \text{ on } \{\rho < \infty\},\quad  \P\text{-}a.s.,  \\ \label{eq:regularInLemma2}
    & \rho = \inf \left\lbrace s \in [0,\infty)\,:\, H(M_s) = 0 \right\rbrace, \quad  \P\text{-}a.s.     
  \end{align}  
\end{proof}

In order to create a better understanding of the time-change $\tau$ (defined by the Carath\'eodory differential equation~\eqref{eq:ode time-change}) and the assumptions of Lemma~\ref{lem:timechange}, two remarks are provided for the special case of Brownian motion.

\begin{remark}
  As we have seen in the proof of Lemma~\ref{lem:timechange}, \eqref{eq:recurrent} or Assumption~\ref{ass:SigmaBounded} is required to ensure that the random time $\tau(\thor)$ is $\P$-almost surely finite. More precisely, we used 
  \begin{equation}\label{eq:delta finite}
    \mathcal{T}(t)= \int_0^t \sigma(\mathcal{T}(s),M_s)^{-1} \dd s = \thor 
  \end{equation}
  for some finite $t\in [0,\infty)$, $\P$-a.s. For example, if $M$ is a Brownian motion under $\P$, it is possible to verify condition~\eqref{eq:delta finite} a posteriori by the uniqueness in law of the time-changed process $X_s:=M_{\tau(s)}$ for $s\in [0,\thor]$ as done by Bass~\cite{Bass1983}.   
\end{remark}

\begin{remark}
  Let us consider the time-homogeneous case where $M$ is a Brownian motion under~$\P$. Suppose that $\sigma = H$ and we have a unique finite solution $(\tau(t))_{t\in[0,\thor]}$ to the differential equation~\eqref{eq:ode time-change}. Then the time-changed process $X_t:=M_{\tau(t)}$ is a weak solution to the Brownian stochastic differential equation (SDE) 
  \begin{equation*} 
    \d X_t = \sqrt{H(X_t)} \dd M_t, \quad X_0 = x_0\in \R, \quad t\in [0,\thor].
  \end{equation*}
  From Engelbert and Schmidt~\cite{Engelbert1985} (see also \cite[Chap.~5, Thm.~5.7]{Karatzas1991a}) we know that a solution to this SDE exists and uniqueness in law holds for this SDE if and only if $\{ x \in \mathbb{R}\,:\, H(x) = 0 \} = I(H)$. Such a function $H$ is regular for (the law of) $M$ under $\P$ in the sense of Definition~\ref{def:regular}, see Example~\ref{ex:regularAt0isregular}. For example, if $H(x) = |x|^\alpha$ for $\alpha \in (0,1)$, then there is no uniqueness in law. However, the so-called fundamental solution in sense of \cite{Engelbert1985} is unique in law: The weak solution of the SDE satisfying $H(X_s(\omega)) > 0$ $\text{Leb} \otimes \P$-almost surely is unique in law, i.e. a solution that \textit{does not spend time at the zeros of $H$}. This is exactly how we construct the time-change: No time is spent at the zeros of $H$ until the first time at which $1/H$ is not integrable anymore, then we stop.
\end{remark}

\subsection{Solution to the Associated Martingale Problem}

In the next lemma we link the martingale problem for the given process $M$ to the martingale problem for the time-changed process~$X$.

\begin{lemma}\label{lem:time-change-process}
  Let $\sigma$, $M$ and $(\mathcal{F}_t)_{t \geq 0}$ be given as in Lemma~\ref{lem:timechange}. For $(\tau(t))_{t \in [0,\thor]}$ as in Lemma~\ref{lem:timechange} the process $(X_t)_{0 \leq t \leq \thor}$ is defined by $X_t := M_{\tau(t)}$. 
  Suppose that for some $f$, $g \in C_0(E)$ the process 
  \begin{equation*} 
    M^{f,g}_t := f(M_t) - f(M_0) - \int_0^t g(M_s) \dd s,\quad t\in [0,\thor],
  \end{equation*}
  is an $(\mathcal{F}_t)$-martingale and $\sigma g$ is bounded. Then the process $(\tilde{M}_s^{f,g})_{0 \leq s \leq \thor}$, given by
  \begin{equation*}  
    \tilde{M}_t^{f,g} := f(X_t) - f(X_0) - \int_0^t \sigma(s,X_s)g(X_s) \dd s,\quad t\in [0,\thor],
  \end{equation*}
  is a martingale w.r.t. the right-continuous completion of the filtration generated by $X$. 
\end{lemma}

\begin{proof}
  First observe that $\tilde{M}_t^{f,g} = M^{f,g}_{\tau(t)}$ for $t\in [0,\thor]$. Indeed, since $\tau(\cdot)$ is monotone and absolutely continuous on $[0,\thor]$ and the function $s \mapsto g(M_s)$ is integrable on $[0,\tau(\thor)]$, $\P$-a.s., a change of variables (cf. \cite[Cor.~3.57]{Leoni2009}) leads to 
  \begin{equation*} 
    \int_0^{\tau(t)} g(M_s) \dd s = \int_0^t g(M_{\tau(s)})\sigma(s,M_{\tau(s)}) \dd s  = \int_0^t \sigma(s,X_s) g(X_s) \dd u,\quad t \in [0,\thor],
  \end{equation*}
  and thus $\tilde{M}_t^{f,g} = M^{f,g}_{\tau(t)}$.
  
  Therefore, it is sufficient to verify that $(M^{f,g}_{\tau(t)})_{t\in [0,\thor]}$ is a martingale. For this purpose we rely on the optional sampling theorem (see e.g. \cite[Chap.~2, Thm.~2.13]{Ethier1986a}) and check its conditions: Because $f\in C_0(E)$ and $\sigma g$ is bounded, there exists a constant $C:=C(f,g) > 0$ with $|\tilde{M}_t^{f,g}| \leq C$ for $t\in [0,\thor]$ and in particular $\sup_{t \in [0,\thor]} \E[|\tilde{M}_t^{f,g}|] < \infty$. Since $\tau(s)$ is finite for every $s\in [0,\thor]$, $\P$-a.s., and since for every $T < \tau(s)$ there exists $\tilde{s} \in [0,s)$ such that $\tau(\tilde{s}) = T$, we have  $|M^{f,g}_T| = |M^{f,g}_{\tau(\tilde{s})}| = |\tilde{M}^{f,g}_{\tilde{s}}| \leq C$ and 
  \begin{equation*}
    \lim_{T \to \infty} \E\big [|M^f_T|\mathbbm{1}_{\{\tau(s) > T\}}\big] \leq C \lim_{T \to \infty} \P(\tau(s) > T) = 0.
  \end{equation*} 
  Hence, for $u,v \in [0,\thor]$ with $u \leq v$ the optional sampling theorem gives 
  \begin{equation*} 
    \E\big [M^{f,g}_{\tau(v)} \big| \mathcal{F}_{\tau(u)}\big] = M^{f,g}_{\tau(u)},
  \end{equation*}
  which means that $(\tilde{M}_t^{f,g})_{t\in [0,\thor]}$ is an $(\mathcal{F}_{\tau(t)})$-martingale. Because $\tilde{M}^{f,g}_t$ is measurable with respect to $\sigma(X_s\,:\, s\leq t)$ for $t\in [0,\thor]$, $\tilde{M}^{f,g}$ is also a martingale with respect to the usual augmentation of the filtration generated by $X$, see e.g. \cite[Thm.~II.2.8]{Revuz1999} or \cite[Lem.~II.67.10]{Rogers2000}.
\end{proof}

Based on the previous Lemma~\ref{lem:time-change-process}, the time-changed process~$X$ is a solution to the ``time-changed'' martingale problem and the marginal distributions of $X$ satisfy the corresponding Fokker--Planck equation:

\begin{proposition}\label{prop:density time changed}
  Let $\mu_0 \in \mathcal{P}(E)$, $\mathcal{D} \subset C_0(E)$ and $\mathcal{A}\colon \mathcal{D} \to C_0(E)$ be linear. Let $\sigma$ and $M$ be given as in Lemma~\ref{lem:timechange} and assume in addition that $M$ is a solution on $(\Omega,\mathcal{F},\mathbb{P})$ to the RCLL-martingale problem for $(\mathcal{A},\mu_0)$ and either Assumption~\ref{ass:Recurrence}~(ii) or Assumption~\ref{ass:SigmaBounded} holds. Let us denote by $p(t,\cdot)$ the law of $X_t = M_{\tau(t)} = M_{\int_0^t \sigma(s,X_s) \dd s}$ as constructed in Lemma~\ref{lem:time-change-process} with $p(0,\cdot) = \mu_0$ and $\tau(t):= \tau(\thor)$ for $t > \thor$, where $\thor$ is as in Assumption~\ref{ass:sigma}. Then one has: 
  \begin{itemize}
    \item $X$ is a solution to the (time-inhomogeneous) $D_E[0,\infty)$-martingale problem for $(\sigma \mathcal{A}, \mu_0)$,
    \item for any $g \in B([0,\infty)\times E)$ the function $s \mapsto \int_E g(s,x)\, p(s,\d x)$ is measurable, 
    \item $(p(s,\d x))_{s \in [0,\thor]}$ satisfies the Fokker--Planck equation, i.e., for any $f \in \mathcal{D}$,
          \begin{equation}\label{eq:KolmogorovEqnAux}
            \int_E f(x)\, p(t,\d x) - \int_E f(x) \,\mu_0(\d x) = \int_0^t \int_E \sigma(s,x) \mathcal{A}f(x)\, p(s,\d x) \dd s,\quad t \in [0,\thor].
          \end{equation}
  \end{itemize}
\end{proposition}

\begin{proof}
  If $f \in \mathcal{D}$, then $\sigma \mathcal{A} f$ is bounded by Assumption~\ref{ass:Recurrence}~(ii) or~\ref{ass:SigmaBounded}. Combining this with our assumption that $M$ is a solution to the RCLL-martingale problem for $(\mathcal{A},\mu_0)$ and with Lemma~\ref{lem:time-change-process}, we obtain that
  \begin{equation*} 
    \tilde{M}_t^f := f(X_t) - f(X_0) - \int_0^t \sigma(s,X_s) \mathcal{A} f (X_s) \dd s ,\quad t \geq 0,
  \end{equation*}
  is a martingale. In particular, one has $\E[\tilde{M}_t^f] = 0$ for all $t \in [0,\thor]$. Since $\sigma \mathcal{A} f$ is bounded, applying Fubini's theorem yields~\eqref{eq:KolmogorovEqnAux}. Finally, $X\colon \Omega \times [0,\infty) \to E$ is measurable and thus so is $(\omega,s) \mapsto (s,X_s(\omega))$. Hence, for $g \in B([0,\infty) \times E)$ also $(\omega,s) \mapsto g(s,X_s(\omega))$ is measurable, and so, by the measurability statement in Fubini's theorem, also $s \mapsto \int_E g(s,x)\, p(s,\d x)$ is measurable.
\end{proof}

\section{A Uniqueness Result for the Fokker--Planck Equation with Degenerate Coefficients}\label{sec:uniqueness}

If $X$ is a solution to the martingale problem for $(\sigma \mathcal{A}, \mu_0)$ and $p(t,\cdot)$ is the law of $X_t$, then according to the proof of Proposition~\ref{prop:density time changed},
\begin{equation}\label{eq:measurabilityKolmogorovEqn}
  \text{ for any } g \in B([0,\infty)\times E), \, s \mapsto \int_E g(s,x)\, p(s,\d x) \text{ is measurable}
\end{equation}
and for all $f$ nice enough it holds that
\begin{equation}\label{eq:KolmogorovEqn}
  \int_E f(x)\, p(t,\d x) - \int_E f(x) \,\mu_0(\d x) = \int_0^t \int_E \sigma(s,x) \mathcal{A}f(x)\, p(s,\d x) \dd s,\quad t \in [0,\thor].
\end{equation}
Conversely, one may ask if solutions $(p(t,\cdot))_{t \in [0,t_0]}$ to \eqref{eq:KolmogorovEqn} can arise differently. In this section, we provide sufficient conditions which guarantee that the Fokker--Planck equation~\eqref{eq:KolmogorovEqn} (also called Kolmogorov forward equation) uniquely characterizes the law of~$X$, i.e. the one-dimensional marginal laws of~$X$ are the only family of probability measures that satisfy \eqref{eq:KolmogorovEqn} for a large class of functions~$f$. More precisely, we prove the following result on (existence and) uniqueness result to the Fokker--Planck equation for time-inhomogeneous operators:

\begin{theorem}\label{thm:KolmogorovUniqueness} 
  Suppose $\sigma$ and $(\mathcal{D},\mathcal{A})$ satisfy Assumptions~\ref{ass:sigma}, \ref{ass:FellerProcess}, \ref{ass:Hregular} and either \ref{ass:Recurrence} or \ref{ass:SigmaBounded}. Let $\thor$ be as in Assumption~\ref{ass:sigma} and $\mu_0 \in \mathcal{P}(E)$. Then existence and uniqueness hold for~\eqref{eq:KolmogorovEqn}: 
  \begin{enumerate}
    \item[(i)] There exists a family of probability measures~$(p(t,\cdot))_{0\leq t \leq \thor}$ on $E$ which satisfies~\eqref{eq:measurabilityKolmogorovEqn} and~\eqref{eq:KolmogorovEqn} for all $f \in \mathcal{D}$ and $p(0,\cdot) = \mu_0$.
    \item[(ii)] If $(q(t,\cdot))_{0\leq t \leq \thor}$ and $(p(t,\cdot))_{0\leq t \leq \thor}$ are two families of probability measures on $E$ which both satisfy~\eqref{eq:measurabilityKolmogorovEqn} and~\eqref{eq:KolmogorovEqn} for all $f \in \mathcal{D}$ and $q(0,\cdot) = \mu_0 = p(0,\cdot)$, then $q(s,\cdot) = p(s,\cdot)$ for all $s \in [0,\thor]$.
  \end{enumerate}
\end{theorem}

While the existence part (Theorem~\ref{thm:KolmogorovUniqueness}~(i)) follows immediately from Proposition~\ref{prop:density time changed}, the rest of the section is devoted to prove the uniqueness result in Theorem~\ref{thm:KolmogorovUniqueness}. 

As we will see, existence and uniqueness of solutions to the time-inhomogeneous Fokker--Planck equation~\eqref{eq:KolmogorovEqn} is closely related to existence and uniqueness of solutions to the martingale problem for the time-homogeneous operator $\sigma \mathcal{A} + \partial_t$ on $C_0([0,\infty)\times E)$ defined in equation~\eqref{eq:InhomogeneousGenerator} below. We show that the martingale problem for this operator is well-posed and the associated time-homogeneous Fokker--Planck equation determines the marginal laws of the solution uniquely. 
\smallskip

In the present context, mainly two difficulties arise: Firstly, $\sigma$ is only locally bounded, time-inhomogeneous and $\{(t,x) \in [0,\infty)\times E \, :\, \sigma(t,x) = 0\} \neq \emptyset$.  Secondly, even if well-posedness for the martingale problem associated to $\sigma \mathcal{A}$ can be established, it is not automatic that any solution to the Fokker--Planck equation corresponds to a solution to the martingale problem for~$\sigma \mathcal{A}$. 
\smallskip
  
The proof of Theorem~\ref{thm:KolmogorovUniqueness}~(ii) relies on the following theorem on uniqueness for the Fokker--Planck equation corresponding to time-homogeneous operators cited\footnote{\label{fn:pregenerator} More precisely, instead of our hypothesis~(iii), in \cite{Kurtz1998} the weaker requirement that $\mathcal{L}$ is a \textit{pre-generator} is imposed. However, as explained in \cite[Sec.~2]{Kurtz1998} or \cite[Remark~1.1]{Kurtz2001}, (iii) implies that $\mathcal{L}$ is a pre-generator.} from \cite{Kurtz1998} (see also \cite[Thm.~4.1]{Bhatt1993}). 

\begin{theorem}[{\cite[Thm.~2.6~(c)]{Kurtz1998}}]\label{thm:KurtzResult}
  Let $(E_0,d_0)$ be a locally compact, complete, separable metric space, $D(\mathcal{L}) \subset C_b(E_0)$ and $\mathcal{L}\colon D(\mathcal{L}) \to C_b(E_0)$ be linear. Let $\nu \in \mathcal{P}(E)$ and suppose that 
  \begin{itemize}
    \item[(i)]$D(\mathcal{L})$ is an algebra and separates points,
    \item[(ii)] there exists a countable subset $\{h_k \,:\, k \geq 1\} \subset D(\mathcal{L})$ such that 
         \begin{equation}\label{eq:separabilityOfDomain}
           \text{bp-closure}(\text{span}(\{(h_k,\mathcal{L} h_k) \, :\, k \geq 1 \})) \supset \{(h,\mathcal{L} h) \, :\, h \in D(\mathcal{L}) \} ,
         \end{equation}
    \item[(iii)] for each $y \in E_0$, there exists a RCLL-solution to the martingale problem for $(\mathcal{L},\delta_y)$,
    \item[(iv)] uniqueness holds for the martingale problem for $(\mathcal{L},\nu)$.
  \end{itemize}
  Then uniqueness holds for the Fokker--Planck equation for $(\mathcal{L},\nu)$: Suppose $\{\nu_t\}_{t \geq 0} \subset \mathcal{P}(E_0)$ is such that
  \begin{equation}\label{eq:timeInhomogeneousmeasurabilityKolmogorovEqn}
    \text{ for any } g \in B(E_0) \, , \, s \mapsto \int_{E_0} g(y)\, \nu_s(\d y) \text{ is measurable}
  \end{equation}
  and
  \begin{equation}\label{eq:timeInhomogeneousPIDE}
    \int_{E_0} h \dd \nu_t = \int_{E_0} h \dd \nu + \int_0^t \int_{E_0} \mathcal{L} h \dd\nu_s  \dd s ,  \quad t \geq 0,
  \end{equation}
  for all $h \in D(\mathcal{L})$. If $\{\mu_t\}_{t \geq 0} \subset \mathcal{P}(E_0)$ also satisfies \eqref{eq:timeInhomogeneousmeasurabilityKolmogorovEqn} and \eqref{eq:timeInhomogeneousPIDE}, then $\mu_t = \nu_t $ for all $t \geq 0$. 
\end{theorem}

The rest of Section~\ref{sec:uniqueness} is devoted to the proof of Theorem~\ref{thm:KolmogorovUniqueness}~(ii). The argument is split into three parts and we will only get to the actual proof in the third part. The procedure is as follows:
\begin{itemize}
  \item In Section~\ref{subsec:InhomogGenerator}  the time-inhomogeneous problem is put into the time-homogeneous setup by including time as an additional state variable. The associated generator $\mathcal{L}$ is defined in \eqref{eq:InhomogeneousGenerator}. 
  \item In Section~\ref{subsection:WellposednessMartingaleProblem} well-posedness of the martingale problem for $\mathcal{L}$ is proved.
  \item In Section~\ref{subsec:WellPosednessFOkkerPlanck} the results from Section~\ref{subsec:InhomogGenerator} and \ref{subsection:WellposednessMartingaleProblem} are used to show that Theorem~\ref{thm:KurtzResult} can indeed be applied to prove Theorem~\ref{thm:KolmogorovUniqueness}~(ii).
\end{itemize}

\subsection{Reducing the Problem to the Time-Homogeneous Setup}\label{subsec:InhomogGenerator} 

Fix $(\mathcal{D},\mathcal{A})$ as in Assumption~\ref{ass:FellerProcess} and a measurable function $\sigma\colon[0,\infty)\times E \to [0,\infty)$. For $f \in \mathcal{D}$ and $\gamma \in C_c^1[0,\infty)$, define the operator $\mathcal{L}$ by
\begin{equation}\label{eq:InhomogeneousGenerator}
  \mathcal{L}(f\gamma)(t,x) := \gamma(t)\sigma(t,x)\mathcal{A}f(x) + f(x) \gamma'(t), \quad t \in [0,\infty), \, x \in E,
\end{equation}
and linearly extend $\mathcal{L}$ to $D(\mathcal{L}) := \text{span}\{f \gamma \, : \, f \in \mathcal{D}, \gamma \in C_c^1[0,\infty) \} \subset C_0([0,\infty)\times E)$.
\smallskip

The following lemma relates the Fokker--Planck equation~\eqref{eq:KolmogorovEqn} and the martingale problem for the time-inhomogeneous operator~$\sigma \mathcal{A}$ to the Fokker--Planck equation and the martingale problem for the time-homogeneous operator $\mathcal{L}$ on $C_0([0,\infty)\times E)$. Furthermore, sufficient conditions for~\eqref{eq:separabilityOfDomain} are provided.

\begin{lemma}\label{lem:BProperties}
  Suppose $\sigma\colon[0,\infty)\times E \to [0,\infty)$ is measurable, $(\mathcal{D},\mathcal{A})$ is as in Assumption~\ref{ass:FellerProcess} and $(D(\mathcal{L}),\mathcal{L})$ as in~\eqref{eq:InhomogeneousGenerator}. Further suppose $\sigma \mathcal{A} f $ is bounded for all $f \in \mathcal{D}$. Let $s_0 \geq 0$, $\mu_0 \in \mathcal{P}(E)$ and define $\sigma_{s_0}(s,x) := \sigma(s_0+s,x)$. Then the following hold:
  \begin{itemize}
    \item[(i)] If $\sigma$ is bounded, then $(D(\mathcal{L}),\mathcal{L})$ is conservative, i.e.~\eqref{eq:defConservative} holds.
    \item[(ii)] If $X$ is a solution to the (time-inhomogeneous) RCLL-martingale problem for $(\sigma_{s_0} \mathcal{A} , \mu_0)$, then $(s_0 + t,X_t)_{t \geq 0}$ is a solution to the RCLL-martingale problem for $(\mathcal{L},\delta_{s_0} \otimes \mu_0)$.
    \item[(iii)] If $\sigma$ is bounded and $(T,X)$ is a solution to the RCLL-martingale problem for $(\mathcal{L},\delta_{s_0} \otimes \mu_0)$, then $X$ is a solution to the (time-inhomogeneous) RCLL-martingale problem for $(\sigma_{s_0} \mathcal{A} , \mu_0)$ and $T$ is indistinguishable from $(s_0+t)_{t \geq 0}$.
    \item[(iv)] Suppose $(p(t,\cdot))_{t \in [0,\thor]}$ satisfies \eqref{eq:measurabilityKolmogorovEqn} and \eqref{eq:KolmogorovEqn} for all $f \in \mathcal{D}$ and define $p(t,\cdot) := p(\thor,\cdot)$ and $\sigma(t,\cdot) := 0$ for $t > \thor$ and for all $t \in [0,\infty)$ the measures $\nu_t := \delta_{t} \otimes p(t,\cdot)$ on $E_0 :=[0,\infty) \times E$. Then
    $\{\nu_t\}_{t \geq 0}$ satisfies \eqref{eq:timeInhomogeneousmeasurabilityKolmogorovEqn} and \eqref{eq:timeInhomogeneousPIDE} for all $h \in D(\mathcal{L})$.
    \item[(v)] Suppose either $\sigma \mathcal{A} f \in C_0([0,\thor] \times E)$ for all $f \in \mathcal{D}$ or $\sigma$ is bounded. Then there exists a countable subset $\{h_k \, :\, k \in \mathbb{N}\} \subset D(\mathcal{L})$ such that \eqref{eq:separabilityOfDomain} holds.
  \end{itemize}
\end{lemma}

\begin{proof}
  \textit{(i)} Note that by Assumption~\ref{ass:FellerProcess}~(i) and \eqref{eq:defConservative}, there exists $\{f_n\}_{n \in \mathbb{N}} \subset \mathcal{D}$ such that $\bplim_{n \to \infty} f_n = 1$ and $\bplim_{n \to \infty} \mathcal{A} f_n = 0$. Furthermore, there exist $\gamma_n \in C_c^1[0,\infty)$, $n \in \mathbb{N}$, such that $\gamma_n(s) = 1$ for $s \in [0,n]$, $\gamma_n(s) = 0$ for $s \in [n+1,\infty)$ and $\sup_{n \in \mathbb{N}} \sup_{s \in [0,\infty)} |\gamma_n'(s)| < \infty$.\footnote{For instance the functions defined for $n \in \mathbb{N}$ by  $\gamma_n(s) = 1$ for $s \in [0,n]$, $\gamma_n(s) = 0$ for $s \in [n+1,\infty)$ and $\gamma_n(s) = 2(s-n)^3-3(s-n)^2+1$ for $s \in [n,n+1]$ satisfy these properties.} In particular, $\bplim_{n \to \infty} \gamma_n' = 0$, $\bplim_{n \to \infty} f_n \gamma_n = 1$ and, since $\sigma$ is bounded, also $\bplim_{n \to \infty} \mathcal{L}(f_n \gamma_n) = 0$. Thus, \eqref{eq:defConservative} holds with $h_n := f_n \gamma_n$.
  \smallskip
  
  \textit{(ii)} By assumption $\sigma \mathcal{A} f$ is bounded for all $f \in \mathcal{D}$, thus $\mathcal{L} h \in B([0,\infty)\times E)$ for all $h \in D(\mathcal{L})$. Therefore, \cite[Chap.~4, Thm.~7.1]{Ethier1986a} implies that $(t,X_t)_{t \geq 0}$ is a solution to the martingale problem for $(D(\mathcal{L}),\mathcal{L}^{\sigma_{s_0}}) $, where $\mathcal{L}^{\sigma_{s_0}}$ is given in \eqref{eq:InhomogeneousGenerator} with $\sigma$ replaced by $\sigma_{s_0}$. Inserting $h = f \tilde{\gamma}$ and $\mathcal{L}^{\sigma_{s_0}}$ in \eqref{eq:InHomMartingaleProblemDef}, this implies that
  \begin{equation}\label{eq:auxEqn4} 
    \tilde{\gamma}(t)f(X_t)-  \tilde{\gamma}(0)f(X_0) - \int_0^t \tilde{\gamma}(s)\sigma(s_0 + s,X_s)\mathcal{A}f(X_s) \dd s - \int_0^t \tilde{\gamma}'(s)f(X_s) \dd s , \quad t \geq 0,
  \end{equation}
  is an $(\mathcal{F}_t^X)_{t \geq 0}$-martingale for all $f\in \mathcal{D}$ and $\tilde{\gamma} \in C_c^1[0,\infty)$. In particular, for given $\gamma \in C_c^1[0,\infty)$, we can use $\tilde{\gamma} := \gamma(\cdot + s_0)$ (which is again in $C_c^1[0,\infty)$) in \eqref{eq:auxEqn4} to see that 
  \begin{equation*} 
    \gamma(s_0+t)f(X_t) - \gamma(s_0) f(X_0) - \int_0^t \mathcal{L}(f\gamma)(s_0+s,X_s) \dd s ,  \quad t \geq 0,
  \end{equation*}
  is an $(\mathcal{F}_t^X)_{t \geq 0}$-martingale for all $f\in \mathcal{D}$ and $\gamma \in C_c^1[0,\infty)$. By linearity, this extends to all $h \in D(\mathcal{L})$ and therefore $(s_0+t,X_t)_{t \geq 0}$ is a solution to the RCLL-martingale problem for $(\mathcal{L},\delta_{s_0} \otimes \mu_0)$.
  \smallskip

  \textit{(iii)} First note that both $A_1:= \{(\gamma,\gamma')\, :\, \gamma \in C_c^1[0,\infty)\}$ and $A_2:= \{(f,\sigma\mathcal{A}f) \, :\, f \in \mathcal{D} \}$ (viewed as subsets of $B([0,\infty) \times E)$) are contained in $\text{bp-closure}(\{(h,\mathcal{L} h) \, :\, h \in D(\mathcal{L}) \})  $. For $A_1$, this follows by fixing $\gamma$ and using $\{f_n\}_{n \in \mathbb{N}} \subset \mathcal{D}$ from the proof of (i) to obtain $\bplim_{n \to \infty} (f_n \gamma) = \gamma$ and $\bplim_{n \to \infty} \mathcal{L}(f_n \gamma) = \gamma'$ since $\sigma$ is bounded. For $A_2$, this follows analogously by using $\gamma_n$ as defined in the proof of (i) and by noting that for any $f \in \mathcal{D}$, $\bplim_{n \to \infty} (\gamma_n f, \mathcal{L}(\gamma_n f)) = (f,\sigma\mathcal{A}f)$. So, if $(T,X)$ is a solution to the RCLL-martingale problem for $(\mathcal{L},\delta_{s_0} \otimes \mu_0)$, then by Remark~\ref{rmk:martingaleProblemBpClosure} 
  \begin{equation}\label{eq:auxEqn6} 
    h(T_t,X_t) - h(T_0,X_0) - \int_0^t g(T_s,X_s) \dd s,  \quad t \geq 0,
  \end{equation}
  is an $(\mathcal{F}^{(T,X)}_t)_{t \geq 0}$-martingale for all $(h,g) \in \text{bp-closure}(\{(h,\mathcal{L} h) \, :\, h \in D(\mathcal{L}) \}) $ and thus in particular for all $(h,g) \in A_1 \cup A_2$. Inserting $(\gamma,\gamma') \in A_1$ in \eqref{eq:auxEqn6} thus yields that 
  \begin{equation*} 
    \gamma(T_t) - \gamma(T_0) - \int_0^t \gamma'(T_s) \dd s,  \quad t \geq 0,
  \end{equation*}
  is an $(\mathcal{F}^{(T,X)}_t)_{t \geq 0}$-martingale and, since it is $(\mathcal{F}^T_t)_{t\geq 0}$-adapted, also a martingale with respect to $(\mathcal{F}^T_t)_{t\geq 0}$. Thus, $T$ is a solution to the RCLL-martingale problem for $(\partial_t,\delta_{s_0})$, where $\partial_t$ has domain $D(\partial_t) := C_c^1[0,\infty)$ and is defined as $\partial_t \gamma := \gamma'$ for $\gamma \in D(\partial_t)$.  However, $(s_0+t)_{t \geq 0}$ is also a solution to the RCLL-martingale problem for $(\partial_t,\delta_{s_0})$ since $\int_0^t \gamma'(s_0+s) \dd s = \gamma(s_0+t) - \gamma(s_0)$ for all $t \geq 0, \gamma \in D(\partial_t)$. By \cite[Chap.~4, Thm.~4.1]{Ethier1986a} uniqueness holds for the martingale problem for $(\partial_t,\delta_{s_0})$ and in particular $T$ is indistinguishable from $(s_0+t)_{t \geq 0}$.\footnote{To check the assumptions of \cite[Chap.~4, Thm.~4.1]{Ethier1986a} in more detail (see \cite{Ethier1986a} for unexplained definitions), note that $[0,\infty)$ is locally compact, separable, $D(\partial_t)$ is dense in $C_0[0,\infty)$ and $C_0[0,\infty)$ is convergence determining (see \cite[Chap.~3, Prop.~4.4]{Ethier1986a}), hence separating. Furthermore, whenever $\gamma \in D(\partial_t)$, $t^* \geq 0$ satisfy $\gamma(t^*) = \sup_{t\geq 0} \gamma(t)$, then $\gamma'(t^*) = 0$. Thus $\partial_t$ satisfies the positive maximum principle and is hence dissipative by \cite[Chap.~4, Lem.~2.1]{Ethier1986a}. Finally, fix $\lambda > 0$, then for any $g \in C_c^1[0,\infty)$, the function $\gamma(t):=\exp(\lambda t)\int_t^\infty g(s)\exp(-\lambda s) \dd s $ satisfies $\gamma \in D(\partial(t))$ and $\lambda \gamma - \partial_t \gamma = g$ so that the range of the operator $\lambda - \partial_t$ is $C_c^1[0,\infty)$ and in particular dense in $C_0[0,\infty)$.}

  On the other hand, \eqref{eq:auxEqn6} is a martingale for all $(h,g) \in A_2$ as deduced above and so for each $f \in \mathcal{D}$,
  \begin{equation*}
    f(X_t) - f(X_0) - \int_0^t \sigma(T_s,X_s)\mathcal{A}f(X_s) \dd s, \quad t \geq 0,
  \end{equation*}
  is a martingale. Since $T$ is indistinguishable from $(s_0+t)_{t \geq 0}$, the claim follows. 
  \smallskip

  \textit{(iv)} First notice that \eqref{eq:KolmogorovEqn} actually holds for all $t \geq 0$ since $p(t,\cdot) = p(\thor,\cdot)$ and $\sigma(t,\cdot) = 0$ for $t > \thor$. Moreover, for any $f \in \mathcal{D}$, $\sigma \mathcal{A} f$ is bounded by assumption and so the function $s \mapsto \int_E \sigma(s,x) \mathcal{A}f (x) \, p(s,\d x)$ is measurable by \eqref{eq:measurabilityKolmogorovEqn} and bounded. Thus, from \eqref{eq:KolmogorovEqn} we see that $t \mapsto F(t) := \int_E f(x)\, p(t,\d x) $ is absolutely continuous (c.f. \cite[Lem.~3.31]{Leoni2009}) with $F'(t) = \int_E \sigma(t,x) \mathcal{A}f(x) \, p(t,\d x)$ for a.e. $t \geq 0$. Hence, for any $f \in \mathcal{D}$ and $\gamma \in C_c^1[0,\infty)$, we may integrate by parts (see \cite[Cor.~3.37]{Leoni2009}) to obtain
  \begin{align*} 
    \gamma(t) F(t) - \gamma(0) F(0) & = \int_0^t \gamma(s) F'(s) \dd s + \int_0^t \gamma'(s) F(s) \dd s  \\
    & = \int_0^t \int_E \gamma(s) \sigma(s,x)\mathcal{A}f(x)\, p(s,\d x) \dd s + \int_0^t \int_E \gamma'(s) f(x)\, p(s,\d x) \dd s \\ & \stackrel{\mathmakebox[\widthof{=}]{\eqref{eq:InhomogeneousGenerator}}}{=} \int_0^t \int_E \mathcal{L}(\gamma f)(s,x)\, p(s,\d x) \dd s = \int_0^t \int_{E_0} \mathcal{L}(\gamma f) \dd \nu_s \dd s
  \end{align*}
  for any $t \geq 0$. But $\gamma(t)F(t) = \int_{E_0} (f\gamma) \dd \nu_t$ by definition, thus $\{\nu_t\}_{t \geq 0}$ satisfies \eqref{eq:timeInhomogeneousPIDE} for all $h = f \gamma$ and by linearity also for all $h \in D(\mathcal{L})$. Finally, note that by definition of $\{\nu_t\}_{t \geq 0}$ the integral in~\eqref{eq:timeInhomogeneousmeasurabilityKolmogorovEqn} is the same as in~\eqref{eq:measurabilityKolmogorovEqn} and so the result follows.
  \smallskip
   
  \textit{(v)} Assume first $\sigma \mathcal{A} f \in C_0([0,\thor] \times E)$ for all $f \in \mathcal{D}$. Set $E_0 :=[0,\thor] \times E$ and note that the spaces $C_0(E_0)$ and $C_0(E_0)\times C_0(E_0)$ are separable since $E$, $[0,\thor]$ and $E_0$ are separable and because products of separable spaces are separable. For $f \in \mathcal{D}$ and $\gamma \in C^1[0,\thor]$ our assumption and $\mathcal{D} \subset C_0(E)$ imply $f \gamma' \in C_0(E_0)$  and $\gamma \sigma \mathcal{A} f \in C_0(E_0)$. Hence, also $\mathcal{L}(f\gamma) \in C_0(E_0)$ and by linearity, $\mathcal{L} h \in C_0(E_0)$ for any $h \in D(\mathcal{L})$. Setting $G_0:=\{(h,\mathcal{L} h) \, :\, h \in D(\mathcal{L}) \}$, this shows $G_0 \subset C_0(E_0) \times C_0(E_0)$. Since the latter space is separable (as argued above) and any subspace of a separable metric space is separable, we conclude that there exists $H_0 \subset G_0$, $H_0$ countable, such that each $(h,\mathcal{L}h) \in G_0$ is the limit in sup-norm of a sequence in $H_0$. In particular, $G_0 \subset \text{bp-closure}(H_0)$, i.e. \eqref{eq:separabilityOfDomain} holds.
   
  Secondly, assume that $\sigma$ is bounded. The same separability reasoning as above shows that there exist $\{\gamma_k\}_{k \in \mathbb{N}} \subset C_c^1[0,\infty)$ and $\{f_l\}_{l \in \mathbb{N}} \subset \mathcal{D}$ with the property that for any $\gamma \in C_c^1[0,\infty)$ and $f \in \mathcal{D}$, there exist $\{k_n\}_{n \in \mathbb{N}}$, $\{l_n\}_{n \in \mathbb{N}} \subset \mathbb{N}$ such that $\gamma = \lim_{n \to \infty} \gamma_{k_n}$, $\gamma' = \lim_{n \to \infty} \gamma'_{k_n}$, $f = \lim_{n \to \infty} f_{l_n}$ and $\mathcal{A} f = \lim_{n \to \infty} \mathcal{A} f_{l_n}$ in sup-norm. Since $\sigma$ is bounded, this also implies $\bplim_{n \to \infty} \sigma \mathcal{A} f_{l_n} = \sigma \mathcal{A} f$ and thus $\bplim_{n \to \infty} \gamma_{k_n} f_{l_n}  = \gamma f$ and $\bplim_{n \to \infty} \mathcal{L}(\gamma_{k_n} f_{l_n}) = \mathcal{L}(\gamma f)$. Thus, we have shown
  \begin{equation*}
    \{(f \gamma,\mathcal{L}(f\gamma))\, :\, f \in \mathcal{D}, \gamma \in C_c^1[0,\infty)\} \subset  \text{bp-closure}(\{(\gamma_k f_l , \mathcal{L}(\gamma_k f_l)) \}_{k,l \in \mathbb{N}})
  \end{equation*}
  and by linearity this implies \eqref{eq:separabilityOfDomain}.
\end{proof}

\subsection{Well-Posedness of the Martingale Problem}\label{subsection:WellposednessMartingaleProblem}

In this section, we show that the martingale problem for $(D(\mathcal{L}),\mathcal{L})$, see~\eqref{eq:InhomogeneousGenerator} above, is well-posed. The proof is split into three parts: Existence is established in Proposition~\ref{prop:ExistenceUnbounded}, uniqueness is proved in Proposition~\ref{prop:uniquenessBoundedSigma} under the assumption that $\sigma$ is bounded. Finally, in Proposition~\ref{prop:UniquenessUnbounded} the assumption of boundedness is removed. 
\smallskip

To prove existence, we use the time-change construction from Lemma~\ref{lem:timechange}. Extra work is needed to incorporate the time-inhomogeneity.

\begin{proposition}\label{prop:ExistenceUnbounded}
  Suppose $\sigma$ and $(\mathcal{D},\mathcal{A})$ are as in Theorem~\ref{thm:KolmogorovUniqueness} and $(D(\mathcal{L}),\mathcal{L})$ is defined as in~\eqref{eq:InhomogeneousGenerator}. Then for any $s_0 \geq 0$ and $x_0 \in E$ there exists a solution to the RCLL-martingale problem for $(\mathcal{L},\delta_{s_0} \otimes \delta_{x_0})$.
\end{proposition}

\begin{proof}
  Define $\sigma_{s_0}(s,x) :=\sigma(s_0+s,x)$ for $s \geq 0$, $x \in E$. Assumption~\ref{ass:Recurrence}~(ii) or~\ref{ass:SigmaBounded} implies that~$\sigma \mathcal{A} f$ is bounded for any $f \in \mathcal{D}$ and so by Lemma~\ref{lem:BProperties}~(ii) it suffices to show that there exists a solution to the (time-inhomogeneous) RCLL-martingale problem for $(\sigma_{s_0} \mathcal{A} , \delta_{x_0})$.
  
  If $s_0 \geq \thor$, then $\sigma_{s_0}(s,x)=0$ for all $(s,x) \in (0,\infty)\times E$, since $\sigma(t,\cdot) = 0$ for $t > t_0$. Setting $X_t := x_0$ for $t \geq 0$, by definition (c.f.~\eqref{eq:InHomMartingaleProblemDef}) it follows that $X$ is a solution to the (time-inhomogeneous) RCLL-martingale problem for $(\sigma_{s_0} \mathcal{A}, \delta_{x_0})$.  
  
  If $s_0 < \thor$, set $\tilde{\thor}:=\thor-s_0$ and note that $\sigma_{s_0}$ satisfies Assumption~\ref{ass:sigma} on $[0,\tilde{\thor}]$ (and $\sigma_{s_0}(t,\cdot)=0$ for $t > \tilde{\thor}$), since $\sigma$ satisfies Assumption~\ref{ass:sigma}. Furthermore,  $\sigma_{s_0}$ satisfies $\sigma_{s_0}\mathcal{A}f \in C_0([0,\tilde{\thor}]\times E)$ for all $f \in \mathcal{D}$ or $\sigma_{s_0}$ is bounded, since $\sigma$ satisfies Assumption~\ref{ass:Recurrence}~(ii) or \ref{ass:SigmaBounded}. Let $M$ denote the coordinate process on $D_E[0,\infty)$ and $\P=P_{x_0}$ (as defined in Assumption~\ref{ass:FellerProcess}).  Then $\sigma_{s_0}$, $(\mathcal{D},\mathcal{A})$ and $M$ satisfy the assumptions of Proposition~\ref{prop:density time changed}, which implies that there exists a solution to the RCLL-martingale problem for $(\sigma_{s_0} \mathcal{A},\delta_{x_0})$.  
\end{proof}

The next step is to prove uniqueness under the assumption that $\sigma$ is bounded. Combined with Proposition~\ref{prop:ExistenceUnbounded}, well-posedness of the RCLL-martingale problem for $(D(\mathcal{L}),\mathcal{L})$ follows. 
\smallskip

The main idea of the proof is to show that any solution $\tilde{X}$ to the RCLL-martingale problem for $(\mathcal{L},\delta_{(s_0,x)})$ can be written as a time-change
\begin{equation*}
  \tilde X_t = M_{\int_0^t \sigma(s_0+u,\tilde X_u) \dd u},\quad t\geq 0,\, \P\text{-}a.s., 
\end{equation*} 
for $M$ which is a solution to the martingale problem for $(\mathcal{A},\delta_x)$. Corollary~\ref{cor:MeasurabilityofTimeChange} and Assumption~\ref{ass:FellerProcess}~(ii) then allow us to conclude uniqueness. 
\smallskip

Note that if $\sigma(t,x)$ did not depend on $t$, the proof could be simplified significantly by relying on \cite[Chap.~6, Thm.~1.4]{Ethier1986a}. 

\begin{proposition}\label{prop:uniquenessBoundedSigma}
  Suppose $\sigma$ and $(\mathcal{D},\mathcal{A})$ satisfy Assumptions~\ref{ass:sigma}, \ref{ass:FellerProcess}, \ref{ass:SigmaBounded} and \ref{ass:Hregular}. Define $(D(\mathcal{L}),\mathcal{L})$ as in \eqref{eq:InhomogeneousGenerator}, then for each $\nu \in \mathcal{P}([0,\infty)\times E)$ the RCLL-martingale problem for $(\mathcal{L},\nu)$ is well-posed. 
\end{proposition}

\begin{proof}
  Firstly, note that it suffices to show that for each $(s_0,x_0) \in [0,\infty)\times E$ the RCLL-martingale problem for $(\mathcal{L},\delta_{(s_0,x_0)})$  is well-posed: If this is established, we can combine \cite[Thm.~2.1]{Bhatt1993} and Lemma~\ref{lem:BProperties}~(v) to conclude that also for any $\nu \in \mathcal{P}([0,\infty)\times E)$ the RCLL-martingale problem for $(\mathcal{L},\nu)$  is well-posed.

  From Proposition~\ref{prop:ExistenceUnbounded} it follows that for any $(s_0,x_0) \in [0,\infty)\times E$ there exists a solution to the RCLL-martingale problem for $(\mathcal{L},\delta_{(s_0,x_0)})$. In order to prove the current proposition, by the above it is therefore sufficient to prove that for any $s_0 \in [0,\infty)$ and any $x_0 \in E$ uniqueness holds for the RCLL-martingale problem for $(\mathcal{L},\delta_{s_0} \otimes \delta_{x_0})$. This will now be established. 

  Set $\mu_0 := \delta_{x_0}$ and suppose $(T,X)$ is a solution RCLL-martingale problem for $(\mathcal{L},\delta_{s_0} \otimes \mu_0)$ defined on some probability space $(\tilde{\Omega},\tilde{\mathcal{F}},\tilde{\P})$. By Lemma~\ref{lem:BProperties}~(iii) it follows that, $\tilde{\P}$-a.s., $T_t = t + s_0$ for all $t \geq 0$ and that $X$ is a solution to the (time-inhomogeneous) RCLL-martingale problem for $(\sigma(s_0+\cdot)\mathcal{A},\mu_0)$, i.e. for each $f \in \mathcal{D}$ the process
  \begin{equation}\label{eq:auxEqn2}
    M_t^f : = f(X_t) - f(X_0) - \int_0^t \sigma(s_0+s,X_s)\mathcal{A}f(X_s) \dd s , \quad t \geq 0,
  \end{equation}
  is a martingale. In the following, we show that this implies that $X$ solves the time-change equation~\eqref{eq:timeChangeGeneral} for some process $M$ that is a solution to the RCLL-problem for $(\mathcal{A},\mu_0)$. By the uniqueness result for the time-change equation in Corollary~\ref{cor:MeasurabilityofTimeChange}, it then follows that the law of $X$ is determined by $\sigma$ and the law of $M$. Since uniqueness holds for the RCLL-martingale problem for $(\mathcal{A},\mu_0)$, it then follows that the law of $X$ is uniquely determined by $\sigma$ and $(\mathcal{A},\mu_0)$ and thus the claim follows.
  \smallskip
  
  Before we start, let us consider the case $s_0 \geq \thor$. Since $\sigma(t,x) = 0 $ for all $t > \thor, x \in E$, the integral term in \eqref{eq:auxEqn2} vanishes for any $f \in \mathcal{D}$. In particular, $f(X)$ is a martingale for any $f \in \mathcal{D}$. Combining this with our Assumption~\ref{ass:FellerProcess}~(i) that $\mathcal{D}$ is dense in $C_0(E)$, this implies (see \cite[Chap.~3, Ex.~7]{Ethier1986a}) that $X_t = X_0$, $\tilde{\P}$-a.s., for any $t \geq 0$. However, $X$ is RCLL and thus $X$ is constant $\tilde{\P}$-a.s. In particular, the law of $(T,X)$ is uniquely determined.
  
  Thus, we may assume $s_0 < \thor$ and in analogy to \eqref{eq:ode time-change} define 
  \begin{equation*}
    \tau(t) := \int_0^t \sigma(s_0 + u, X_u ) \dd u,
  \end{equation*} 
  for each $t \geq 0$, and set $\mathcal{T}(u) := \inf\{t \geq 0 \, : \, \tau(t) \geq u \} $ for each $u \geq 0$ and $Y_u := X_{\mathcal{T}(u)}$ for $u \leq \tau(\thor-s_0)$. Note that $\sigma$ is bounded and $\sigma(s,\cdot) = 0$ for $s > \thor$, thus $\tilde{\P}$-a.s. $\tau(\thor-s_0) < \infty$ and $\tau(s) = \tau(\thor-s_0)$ for all $s > \thor-s_0$. 
  We now claim that 
  \begin{itemize}
    \item[(i)] with probability one, $X$ is constant on any interval $[t,u]$ with $\int_t^u \sigma(s_0+s,X_s) \dd s = 0$,
    \item[(ii)] $\tilde{\P}$-a.s., $X$ satisfies the time-change equation $X_t = Y_{\tau(t)}$ for all $t \geq 0$ and $Y$ is RCLL,
    \item[(iii)] on an extended probability space there exists a process $\tilde{Y}$ that is a solution to the RCLL-martingale problem for $(\mathcal{A},\mu_0)$ that satisfies $\tilde{Y}_u = Y_u $ for $u \leq \tau(\thor-s_0)$ and such that $X = \tilde{Y}_{\tau(t)}$ still holds a.s. Furthermore, this implies the claim.
  \end{itemize} 
  
  To prove (i), we define  
  \begin{equation*} 
    \gamma_t := \inf\{ u > t \, : \, \tau(u) > \tau(t) \}\quad \text{for each } t \geq 0
  \end{equation*}
  so that for any $u > t$,  $\int_t^{\gamma_t \wedge u} \sigma(s_0+s,X_s) \dd s = 0$  and thus (recall $\sigma \geq 0$) also \begin{equation}\label{eq:auxEqn8} 
    \sigma(s_0+s,X_s) = 0\quad \text{ for a.e. }s \in [t,\gamma_t \wedge u],  \, \tilde{\P}\text{-a.s.} 
  \end{equation}
  Furthermore, for $u < t$, $ \{ \gamma_t \leq u \} = \emptyset$ and for $u \geq t$,
  \begin{equation}\label{eq:auxEqn7} 
    \{ \gamma_t < u \} = \{ \tau(u) > \tau(t) \}
  \end{equation}
  since $\tau$ is non-decreasing and continuous (which implies $\tau(\gamma_t) = \tau(t)$). Let us denote by $(\mathcal{F}_t)_{t \geq 0}$ the usual augmentation (in the sense of \cite[Def.~II.67.3]{Rogers2000}) of $(\mathcal{F}_t^X)_{t \geq 0}$. Since $\tau$ is adapted, \eqref{eq:auxEqn7} implies $\{ \gamma_t < u \} \in \mathcal{F}_u$ for any $u \geq 0$ and so $\gamma_t$ is an $(\mathcal{F}_t)_{t \geq 0}$-stopping time. By right-continuity of $X$ and \cite[Lem.~II.67.10]{Rogers2000}, \eqref{eq:auxEqn2} is also a martingale with respect to $(\mathcal{F}_t)_{t \geq 0}$. Hence, we can apply optional sampling to the martingales~\eqref{eq:auxEqn2} to obtain for any $u > t$ and $f \in \mathcal{D}$ that
  \begin{align*}
    0 = \mathbb{E}\big[M^f_{\gamma_t \wedge u} - M^f_t \big|\mathcal{F}_t\big]
    & = \mathbb{E}\left[f(X_{\gamma_t \wedge u}) - \int_t^{\gamma_t \wedge u } \sigma(s_0+s,X_s)\mathcal{A}f(X_s) \dd s \bigg|\mathcal{F}_t\right] - f(X_t) \\ 
    & = \mathbb{E}\left[f(X_{\gamma_t \wedge u}) |\mathcal{F}_t\right] - f(X_t),
  \end{align*}
  where the last step follows from \eqref{eq:auxEqn8}. Since $f \in \mathcal{D}$ was arbitrary and by Assumption~\ref{ass:FellerProcess}~(i) $\mathcal{D}$ is dense in $C_0(E)$, this implies (see \cite[Chap.~3, Ex.~7]{Ethier1986a}) that for fixed $u>t$, $X_{\gamma_t \wedge u} = X_t$, $\tilde{\P}$-a.s. Thus we can find $\Omega_0 \in \tilde{\mathcal{F}}$ such that $\tilde{\P}(\Omega_0) = 1$ and on $\Omega_0$, we have  $X_{\gamma_t \wedge u} = X_t$ for all $u > t \geq 0$ with $t,u \in \mathbb{Q}$. But then on $\Omega_0$ this extends to all $u > t \geq 0$ by a standard argument: for $u > t \geq 0$, we find $\{u_n \} \subset \mathbb{Q}$, $\{t_n\} \subset \mathbb{Q}$ with $u_n \downarrow u$, $t_n \downarrow t$ as $n \to \infty$. Then $\gamma_{t_n} \wedge u_n \downarrow \gamma_t \wedge u$ as $n \to \infty$ and so we can use right-continuity of $X$ for the first and last equality and our choice of $\Omega_0$ for the second equality to obtain 
  \begin{equation}\label{eq:auxEqn9} 
    X_t = \lim_{n \to \infty} X_{t_n} = \lim_{n \to \infty} X_{\gamma_{t_n} \wedge u_n}= X_{\gamma_t \wedge u}.
  \end{equation}
  Thus, if $\omega \in \Omega_0$, $u > t \geq 0$ and $\int_t^u \sigma(s_0+s,X_s(\omega))\dd s = 0$, then $u \leq \gamma_t(\omega)$ and so by \eqref{eq:auxEqn9} indeed $X_t(\omega) = X_{\gamma_t(\omega)\wedge u} = X_u(\omega)$. 
  \smallskip

  To prove (ii), set $\mathcal{T}_{+}(u) := \lim_{v \downarrow u} \mathcal{T}(v)$ and notice $\mathcal{T}_+(u) = \inf\{t \geq 0 \,:\, \tau(t) > u\}$. Then from (i) we get that, $\tilde{\P}$-a.s., $X$ is constant on the interval $[\mathcal{T}(u),\mathcal{T}_{+}(u)]$ for all $u \geq 0$. Hence, from $\mathcal{T}(\tau(t)) \leq t \leq \mathcal{T}_{+}(\tau(t))$ we obtain
  \begin{equation*} 
    X_t = X_{\mathcal{T}(\tau(t))} = Y_{\tau(t)}
  \end{equation*}
  for all $t \geq 0$. Furthermore, $u \downarrow u_0$ implies $\mathcal{T}(u) \downarrow \mathcal{T}_+(u_0)$ and so by right-continuity of $X$ also 
  \begin{equation*}
    Y_u = X_{\mathcal{T}(u)} \to X_{\mathcal{T}_+(u_0)} = X_{\mathcal{T}(u_0)} = Y_{u_0}
  \end{equation*}
  and since $\mathcal{T}$ is left-continuous and $X$ has left-limits, the same reasoning shows that $Y$ also has left-limits. Hence, $Y$ is indeed RCLL.
  \smallskip
  
  For (iii), notice that $\sigma(s_0+s,X_s) = 0$ for $s \in [\mathcal{T}(\tau(\thor-s_0)),\thor-s_0]$ and $s \in [\mathcal{T}(v),\mathcal{T}_+(v)]$ for any $v \geq 0$. Combining this with $ \mathcal{T}(u) \leq t \Leftrightarrow \tau(t) \geq u$, one obtains 
  \begin{align}\label{eq:auxEqn11}
  \begin{split}
    \int_0^{\mathcal{T}(u)\wedge (\thor-s_0)} \sigma(s_0 + s, X_s )\mathcal{A}f(X_s) \dd s & =\int_0^{\mathcal{T}(u\wedge \tau(\thor-s_0))} \sigma(s_0 + s, X_s )\mathcal{A}f(X_s) \dd s \\ & =  \int_0^{\mathcal{T}_+(u\wedge \tau(\thor-s_0))} \sigma(s_0 + s, X_s )\mathcal{A}f(X_s) \dd s, 
    \end{split}
  \end{align}
  for $u \geq 0$, $f \in \mathcal{D}$.
  Since  $\mathcal{T}_+(v) = \inf\{t \geq 0 \,:\, \tau(t) > u\}$, a change of variables as in \cite[Chap.~6, Ex.~12]{Ethier1986a} now allows to rewrite the last expression as  
  \begin{equation}\label{eq:auxEqn10} 
    \int_0^{\mathcal{T}_+(u \wedge \tau(\thor-s_0))} \sigma(s_0+s,X_s)\mathcal{A}f(X_s) \dd s = \int_0^{u \wedge \tau(\thor-s_0)} \mathcal{A}f(X_{\mathcal{T}_+(v)}) \dd v.
  \end{equation}  
  As argued in (ii), $X_{\mathcal{T}_+(v)} = Y_v$ for all $v \geq 0$ and by inserting this in the right-hand side of~\eqref{eq:auxEqn10} and combining with \eqref{eq:auxEqn11}, we get
  \begin{equation}\label{eq:auxEq12}
    \int_0^{\mathcal{T}(u)\wedge (\thor-s_0)} \sigma(s_0 + s, X_s )\mathcal{A}f(X_s) \dd s  = \int_0^{u\wedge \tau(\thor-s_0)} \mathcal{A}f(Y_s) \dd s 
  \end{equation}
  for any $u \geq 0$, $f \in \mathcal{D}$. Furthermore, for any $t,u\geq 0$, $\{ \mathcal{T}(u) \leq t \} = \{ \tau(t) \geq u \} \in \mathcal{F}_t$ so that for each $u \geq 0$, $\mathcal{T}(u)$ is a stopping time.  Using this, \eqref{eq:auxEq12} and applying the optional sampling theorem to the martingales in \eqref{eq:auxEqn2}, we therefore get that for any $f \in \mathcal{D}$ the process
  \begin{equation}\label{eq:auxEqn3} 
    M^f_{\mathcal{T}(u) \wedge (\thor-s_0)} = f(Y_{u \wedge \tau(\thor-s_0)}) - f(Y_0) - \int_0^{u \wedge \tau(\thor-s_0)} \mathcal{A}f(Y_s) \dd s, \quad u \geq 0,
  \end{equation}
  is a martingale with respect to the filtration $(\mathcal{F}_{\mathcal{T}(u)\wedge \thor- s_0})_{u \geq 0}$ and thus also with respect to the filtration generated by $Y_{\cdot \wedge \tau(\thor-s_0)}$. Let us denote by $W$ the coordinate process on $D_E[0,\infty)$ and for $(\omega,\omega') \in \tilde{\Omega} \times D_E[0,\infty)$ define
  \begin{equation*}
    \tilde{Y}_u(\omega,\omega') := \begin{cases} 
            Y_{u}(\omega) &\mbox{for } u < \tau(\thor-s_0)(\omega), \\
            W_{u - \tau(\thor-s_0)(\omega)}(\omega') & \mbox{for } u \geq \tau(\thor-s_0)(\omega). 
          \end{cases}
  \end{equation*} 
  From \eqref{eq:auxEqn3} and Lemma~\ref{lem:gluingLemma} below (applied to the process $(Y_{u \wedge \tau(\thor-s_0)})_{u \geq 0}$ and the random variable $\tau(\thor-s_0)$) it follows that the process $\tilde{Y}$ is a solution to the RCLL-martingale problem for $(\mathcal{A},\mu_0)$ under a measure $Q$ with $Q(A\times D_E[0,\infty)) = \tilde{\P}(A)$ for all $A \in \tilde{\mathcal{F}}$ and such that $\tilde{Y}_{s \wedge \tau(\thor-s_0)}= Y_{s \wedge \tau(\thor-s_0)}$ for all $s \geq 0$, $Q$-a.s. Combining this with (ii) and $\tau(\cdot) \leq \tau(\thor-s_0)$, it follows that
  \begin{equation*} 
    X_t = Y_{\tau(t)} = Y_{\tau(t)\wedge \tau(\thor-s_0)} = \tilde{Y}_{\tau(t) \wedge \tau(\thor-s_0)} = \tilde{Y}_{\tau(t)} = \tilde{Y}_{\int_0^t \sigma_{s_0}(s,X_s) \dd s}, \quad t \geq 0, \,Q\text{-a.s.,}
  \end{equation*}
  where $\sigma_{s_0}(s,x):=\sigma(s+s_0,x)$ for $s \geq 0, x \in E$. In particular, $X$ satisfies a time-change equation~\eqref{eq:timeChangeGeneral} (with $M$ replaced by $\tilde{Y}$ and $\sigma$ replaced by $\sigma_{s_0}$). By our assumptions on $\sigma$, $\sigma_{s_0}$ satisfies Assumption~\ref{ass:sigma} on $[0,t_0-s_0]$ (and $\sigma_{s_0}(t,\cdot)=0$ for $t > t_0-s_0$) and \ref{ass:SigmaBounded}. Since uniqueness holds for the RCLL-martingale problem for $(\mathcal{A},\mu_0)$, the law on $D_E[0,\infty)$ of $\tilde{Y}$ under~$Q$ is given as~$P_{x_0}$ and thus, by Assumption~\ref{ass:Hregular}, $H$ is regular for $Q$. Altogether, Corollary~\ref{cor:MeasurabilityofTimeChange} can be applied to $\sigma_{s_0}$ and $\tilde{Y}$, which implies that the law of $X$ under~$Q$ is uniquely determined by~$\sigma_{s_0}$ and~$P_{x_0}$. But the law of~$X$ under~$\tilde \P$ is the same as under~$Q$ and so the claim follows.
\end{proof}

For the well-posedness of the RCLL-martingale problem (see Proposition~\ref{prop:uniquenessBoundedSigma}), we used the following auxiliary lemma. As the authors are not aware of a suitable reference, we also present its complete proof here.

\begin{lemma}\label{lem:gluingLemma}
  Let $(E,d)$ be a locally compact, complete, separable metric space, $\mathcal{D} \subset C_0(E)$ and $\mathcal{A}\colon\mathcal{D} \to C_0(E)$ linear. Suppose that the $D_E[0,\infty)$-martingale problem for $(\mathcal{A},\mu)$ is well-posed for any $\mu \in \mathcal{P}(E)$. Let $(\Omega,\mathcal{F},\mathbb{P})$ be a probability space, $\tau$ be a $[0,\infty)$-valued random variable, and $Z$ be an $E$-valued RCLL process on $(\Omega,\mathcal{F},\mathbb{P})$ with $Z_u = Z_{u \wedge \tau}$ for all $u \geq 0$, such that 
  \begin{equation*} 
    f(Z_u) - f(Z_0) - \int_0^{u \wedge \tau} \mathcal{A}f(Z_s) \dd s, \quad u \geq 0,
  \end{equation*}
  is a martingale for all $f \in \mathcal{D}$. Let us denote by $X$ the coordinate process on $D_E[0,\infty)$. On $\Omega':=\Omega \times D_E[0,\infty)$ define the process $Y$ via  
  \begin{equation*}
    Y_t(\omega,\omega') := \begin{cases} 
            Z_t(\omega) &\mbox{for } t <\tau(\omega) \\
            X_{t - \tau(\omega)}(\omega') & \mbox{for } t \geq \tau(\omega)
          \end{cases}
  \end{equation*}
  for $(\omega,\omega') \in \Omega'$ and $t \geq 0$. Furthermore, for each $x \in E$, denote by $P_x$ the law of the solution of the RCLL-martingale problem for $(\mathcal{A},\delta_x)$ and by $\mathcal{S}_E$ the Borel $\sigma$-algebra in $D_E[0,\infty)$. Let us define the measure $Q$ on $\mathcal{F} \times \mathcal{S}_E$ by 
  \begin{equation}\label{eq:Qdef}
    Q(A \times C) := \mathbb{E}[\mathbbm{1}_A P_{Z_\tau}(C)]
  \end{equation}  
  for $A \in \mathcal{F}$, $C \in \mathcal{S}_E$ (and extend this to the product $\sigma$-algebra). Then under~$Q$, $Y$ is a solution to the RCLL-martingale problem for $(\mathcal{A},\mu_0)$, where $\mu_0$ is the law of $Z_0$. Furthermore, $Q(A\times D_E[0,\infty)) = \mathbb{P}(A)$ for all $A \in \mathcal{F}$ and $Z_t = Y_{t \wedge \tau}$ for all $t \geq 0$, $Q$-a.s. 
\end{lemma}

\begin{proof}
  Essentially this is \cite[Chap.~4, Lem.~5.16]{Ethier1986a}, the only difference is that we construct~$Y$ on~$\Omega \times D_E[0,\infty)$ (instead of $D_E[0,\infty) \times [0,\infty) \times D_E[0,\infty)$). 
 
  To prove the claim, first notice that by \cite[Chap.~4, Thm.~4.6]{Ethier1986a}, the map $x \mapsto P_x(C)$ is measurable for each $C \in \mathcal{S}_E$ and so $Q$ is indeed well-defined. Furthermore, denoting by $\mu$ the law of $Z_\tau$, also the measure $\tilde{Q}$ defined on product sets as  
  \[
    \tilde{Q}(B \times C) := \int_E \mathbb{E}[\mathbbm{1}_B(Z,\tau)|Z_\tau=x] P_{x}(C)\,\mu(\d x),
  \] 
  for $B \in \mathcal{S}_E \times \mathcal{B}([0,\infty))$, $C \in \mathcal{S}_E$ is well-defined. Therefore, denoting by $(X^{(1)},\eta,X^{(2)})$ the coordinate random variable on $D_E[0,\infty) \times [0,\infty) \times D_E[0,\infty)$, from \cite[Chap.~4, Lem.~5.16, (5.52) and (5.53)]{Ethier1986a} it follows that under $\tilde{Q}$ the process $(\tilde{Y}_t)_{t \geq 0}$ defined  as 
  \begin{equation*}
    \tilde{Y}_t := \begin{cases} 
            X^{(1)}_t &\mbox{for } t < \eta \\
            X^{(2)}_{t-\eta} & \mbox{for } t \geq \eta
          \end{cases}
  \end{equation*}
  is a solution to the RCLL-martingale problem for $(\mathcal{A},\mu_0)$. Thus, it remains to show that the law of~$Y$ under~$Q$ is the same as the law of~$\tilde{Y}$ under~$\tilde{Q}$. 
 
  Firstly note that for $B \in \mathcal{S}_E \times \mathcal{B}([0,\infty))$ and $C \in \mathcal{S}_E$ one obtains
  \begin{equation*}\begin{aligned}
    Q(\{(Z,\tau) \in B\} \times \{X \in C \}) & = \mathbb{E}[\mathbbm{1}_{\{(Z,\tau) \in B\}} P_{Z_\tau}(C)]
    = \mathbb{E}[\mathbb{E}[\mathbbm{1}_B(Z,\tau)|Z_\tau] \mathbb{E}[P_{Z_\tau}(C)|Z_\tau]] 
    \\ & = \tilde{Q}(B \times C),
  \end{aligned}\end{equation*}
  where the last step uses \cite[Chap.~4, Thm.~4.2~(c)]{Ethier1986a}. Hence, $\tilde{Q}$ coincides with the law of $(Z,\tau,X)$ on $D_E[0,\infty)\times [0,\infty) \times D_E[0,\infty)$ under~$Q$. Recalling that $(X^{(1)},\eta,X^{(2)})$ denotes the coordinate random variable on $D_E[0,\infty)\times [0,\infty) \times D_E[0,\infty)$, this implies in particular that for any $0 \leq t_1 < \cdots < t_n$, $n \in \mathbb{N}$, $A_1,\cdots, A_n \in \mathcal{B}(E)$, 
  \begin{align*} 
    Q(Y_{t_1} \in A_1, \dots, Y_{t_n} \in A_n) & = \sum_{I \subset \{1,\dots,n \} } Q(Z_{t_i} \in A_i \, , t_i < \tau  \, \forall i \in I , X_{t_j - \tau} \in A_j , t_j \geq \tau \, \forall j \in I^c ) 
    \\ & = \sum_{I \subset \{1,\dots,n \} } \tilde{Q}(X^{(1)}_{t_i} \in A_i \, , t_i < \eta  \, \forall i \in I , X^{(2)}_{t_j - \eta} \in A_j , t_j \geq \eta \, \forall j \in I^c ) 
    \\ & = \tilde{Q}(\tilde{Y}_{t_1} \in A_1, \dots, \tilde{Y}_{t_n} \in A_n).
  \end{align*}
Thus, it follows that the law of $\tilde{Y}$ under $\tilde{Q}$ is the same as the law of~$Y$ under~$Q$, hence the claim.

  Finally, $Y_u = Z_u$ for $u < \tau$ and so we only need to show $Z_\tau = Y_\tau$, $Q$-a.s. This should be clear but we still give a formal argument: For any $A, B \in \mathcal{B}(E)$ we have, by \eqref{eq:Qdef} and $P_x({X_0 \in B}) = \delta_x(B)$,
  \begin{equation}\label{eq:auxEq13}
    Q(\{Z_\tau \in A\} \times \{X_0 \in B\}) = \mathbb{E}[\mathbbm{1}_A(Z_\tau) P_{Z_\tau}(X_0 \in B)] = \mathbb{P}(Z_\tau \in A \cap B). 
  \end{equation}
  Now denote by $\{x_m \}_{m \in \mathbb{N}}$ a countable dense subset of $E$. Then, for any $n \in \mathbb{N}$, $ m \in \mathbb{N}\setminus{\{n\}}$ and  $k \in \mathbb{N}$ with $d(x_n,x_m)>2/k$, from~\eqref{eq:auxEq13} we have
  \[
    Q(\{Z_\tau \in B_{1/k}(x_n) \} \cap \{X_0 \in B_{1/k}(x_m) \}) = \mathbb{P}(Z_\tau \in B_{1/k}(x_n)\cap B_{1/k}(x_m)) = \mathbb{P}(\emptyset) = 0. 
  \]
  Writing
  \begin{equation*}
    \{Z_\tau \neq Y_\tau \} = \{Z_\tau \neq X_0 \} = \bigcup_{n \in \mathbb{N}} \bigcup_{m \in \mathbb{N}\setminus{\{n\}}} \bigcup_{k \in \mathbb{N}\,:\, d(x_n,x_m)>2/k} \{Z_\tau \in B_{1/k}(x_n) \} \cap \{X_0 \in B_{1/k}(x_m) \},
  \end{equation*}
  we see that $\{Z_\tau \neq Y_\tau \}$ is a countable union of $Q$-null sets and so the claim follows.
\end{proof}

Finally, Proposition~\ref{prop:uniquenessBoundedSigma} is extended in two directions: Firstly, we localize to prove uniqueness for the RCLL-martingale problem when $\sigma$ is unbounded, and secondly, we show that any progressively measurable (but not necessarily RCLL) solution to the martingale problem has an RCLL modification. Note that the last statement is not true in general (as discussed in \cite{Bhatt2003}), but it has to be established to apply Theorem~\ref{thm:KurtzResult}. 
\smallskip

In the proof, the notion of a \textit{stopped martingale problem} is used: Let $(F,d)$ be a complete metric space, $(D(\mathcal{L}),\mathcal{L})$ an operator on $C_b(F)$ and $U$ open in $F$. If $X$ is an $F$-valued, RCLL process, then $\tau := \inf \{t\geq 0 \, : \, X_t \notin U \text{ or } X_{t-} \notin U\}$ is an $(\mathcal{F}_t^X)_{t \geq 0}$-stopping time by \cite[Chap.~2, Prop.~1.5~a)]{Ethier1986a}. $X$ is called a solution to the stopped martingale problem for $(\mathcal{L},\nu_0,U)$ if $X_0 \sim \nu_0$, $X^\tau = X$ and
\begin{equation*} 
  h(X_{\tau \wedge t}) - h(X_0) - \int_0^{\tau \wedge t} \mathcal{L} h (X_s) \dd s,  \quad t \geq 0, 
\end{equation*}
is an $(\mathcal{F}^X_t)_{t\geq 0}$-martingale for any $h \in D(\mathcal{L})$.

\begin{proposition}\label{prop:UniquenessUnbounded}
  Suppose $\sigma$ and $(\mathcal{D},\mathcal{A})$ are as in Theorem~\ref{thm:KolmogorovUniqueness} and $(D(\mathcal{L}),\mathcal{L})$ is defined as in \eqref{eq:InhomogeneousGenerator}. Then for any $\mu_0 \in \mathcal{P}(E)$ the martingale problem for $(\mathcal{L},\delta_{0} \otimes \mu_0)$ is well-posed.
\end{proposition}

\begin{proof} 
  Suppose first $\mu_0 = \delta_{x_0}$ for some $x_0 \in E$. By Proposition~\ref{prop:ExistenceUnbounded} and our assumptions, there exists a solution $Z$ to the RCLL-martingale problem for $(\mathcal{L},\delta_{0} \otimes \mu_0)$. Therefore, it suffices to show that if $\tilde{Z}$ is any (progressively measurable) solution to the martingale problem for $(\mathcal{L},\delta_{0} \otimes \mu_0)$, then $\tilde{Z}$ has the same finite-dimensional marginal distributions as $Z$. In order to prove this, we proceed as follows:
  
  To start, by \cite[Chap.~4, Cor.~3.7]{Ethier1986a} and since $D(\mathcal{L})$ is dense in $C_0(\mathcal{L})$,  $\tilde{Z}$ has a modification (which we also denote by $\tilde{Z}$) with sample paths in $D_{E_0^\Delta}[0,\infty)$, where $E_0^\Delta$ is the one-point compactification of $E_0:=[0,\infty) \times E$. Furthermore, \eqref{eq:martingaleProblemDef} remains vaild for $\tilde{Z}$ and all $h \in D(\mathcal{L})$, where we extend $h$ to $C(E_0^\Delta)$ by $h(\Delta):=0$. By assumption on $E$, there exists $\{V_n\}_{n \in \mathbb{N}} \subset E$ such that for any $n$, $V_n$ is open, $\overline{V_n}$ is compact and $\cup_n V_n = E$. Define $U_n := [0,\infty) \times V_n$, $\tau_n := \inf \{t\geq 0 \, : \, Z_t \notin U_n \text{ or } Z_{t-} \notin U_n\}$ and $\tilde{\tau}_n =:\inf \{t\geq 0 \, : \, \tilde{Z}_t \notin U_n \text{ or } \tilde{Z}_{t-} \notin U_n\} $. Since~$U_n$ is open in~$E_0$, it is also open in $E_0^\Delta$ and thus $\tau_n$ is an $(\mathcal{F}^Z_t)_{t\geq 0}$-stopping time and $\tilde{\tau}_n$ is an $(\mathcal{F}^{\tilde{Z}}_t)_{t\geq 0}$-stopping time.
  
  Suppose Assumption~\ref{ass:Recurrence} holds. Then for any $n \in \mathbb{N}$ there exists $C_n > 0$ such that $|H(x)|\leq C_n$ and $|\tilde{\sigma}(t,x)|\leq C_n$ for all $(t,x) \in U_n$. Set  $\tilde{\sigma}_n(t,x) := \min(\tilde{\sigma}(t,x),C_n)$, $H_n(x) :=\min(H(x),C_n)$ and $\sigma_n := H_n \tilde{\sigma}_n$. In the other case, i.e. if Assumption~\ref{ass:SigmaBounded} holds, set $\sigma_n:= \sigma$. Then, in both cases, $\sigma_n$ is bounded and coincides with $\sigma$ in $U_n$. Define
  \begin{equation*}
    \mathcal{L}^n(f\gamma)(t,x) := \gamma(t)\sigma_n(t,x)\mathcal{A}f(x) + f(x) \gamma'(t), \quad t \in [0,\infty) , \, x \in E,
  \end{equation*}
  for $f \in \mathcal{D}$ and $\gamma \in C_c^1[0,\infty)$ and linearly extended to $D(\mathcal{L}^n) := \text{span}\{f \gamma \, : \, f \in \mathcal{D}, \gamma \in C_c^1[0,\infty) \}$ (and thus $D(\mathcal{L}^n) = D(\mathcal{L}))$.
  
  We now claim that:
  \begin{itemize}
    \item[(i)] $Z^{\tau_n}$ is a solution to the stopped martingale problem for $(\mathcal{L}^n,\delta_{0} \otimes \mu_0,U_n)$ and this solution is unique in law,
    \item[(ii)] $\tilde{Z}^{\tilde{\tau}_n}$ takes values in $D_{E_0}[0,\infty)$ and is also a solution to the stopped martingale problem for $(\mathcal{L}^n,\delta_{0} \otimes \mu_0,U_n)$ and thus, combining this with (i), we get that  the finite-dimensional marginals of $\tilde{Z}^{\tilde{\tau}_n}$ and $Z^{\tau_n}$ agree,
    \item[(iii)] from (ii) it can be deduced that $\tilde{\tau}_n \to \infty$ as $n \to \infty$ and that $\tilde{Z}$ and $Z$ have the same distribution.
  \end{itemize}  
  
  To show (i), notice that $\sigma_n$ is bounded and satisfies Assumption~\ref{ass:sigma} and Assumption~\ref{ass:Hregular} since they hold for $H$ and $\tilde{\sigma}$. In particular, $\sigma_n$ and $(\mathcal{D},\mathcal{A})$ satisfy Assumptions~\ref{ass:sigma}, \ref{ass:FellerProcess}, \ref{ass:SigmaBounded} and~\ref{ass:Hregular} and thus by Proposition~\ref{prop:uniquenessBoundedSigma} the RCLL-martingale problem for $(\mathcal{L}^n,\delta_{0} \otimes \mu_0)$ is well-posed. Therefore, by \cite[Chap.~4, Thm.~6.1]{Ethier1986a} for each $U \subset E_0$ open there exists a unique solution to the stopped martingale problem for $(\mathcal{L}^n,\delta_{0} \otimes \mu_0,U)$. Applying optional sampling to $\tau_n$ and the martingales~\eqref{eq:martingaleProblemDef} and noticing
  \begin{equation}\label{eq:auxEq14} 
    \int_0^{\tau_n \wedge t} \mathcal{L} h(Z_s) \dd s = \int_0^{\tau_n \wedge t} \mathcal{L}^n h(Z_s) \dd s,
  \end{equation}  
  we see that $Z^{\tau_n}$ is a (and hence the unique) solution to the stopped martingale problem for $(\mathcal{L}^n,\delta_{0} \otimes \mu_0,U_n)$.
  \smallskip

  To show~(ii), note that, by definition of $\tilde{\tau}_n$, $\tilde{Z}$ is RCLL and $U_n$-valued on $[0,\tilde{\tau}_n)$. Let us first show that actually $\tilde{\P}$-a.s., for each $n$, $\tilde{Z}_{\tilde{\tau}_n} \in E_0$ (a priori, we could have $\tilde{Z}_{\tilde{\tau}_n} = \Delta$). To do so, note that by applying optional sampling and taking expectations in \eqref{eq:martingaleProblemDef},  we obtain
  \begin{equation}\label{eq:auxEqn1} 
    \E[h(\tilde{Z}_{\tilde{\tau}_n \wedge t})] = \E[h(\tilde{Z}_0)] + \E\left[\int_0^{\tilde{\tau}_n \wedge t} \mathcal{L} h(\tilde{Z}_s) \dd s \right] = \E[h(\tilde{Z}_0)] + \E\left[\int_0^{\tilde{\tau}_n \wedge t} \mathcal{L}^n h(\tilde{Z}_s) \dd s \right]
  \end{equation}
  for all $h \in D(\mathcal{L})$ (with $C_0(E_0)$ extended to $E_0^\Delta$ as above). The second step in \eqref{eq:auxEqn1} follows as in \eqref{eq:auxEq14}. Since $\sigma_n$ is bounded, $(D(\mathcal{L}^n),\mathcal{L}^n)$ is conservative by Lemma~\ref{lem:BProperties}~(i) and so (for any $n \in \mathbb{N}$) there exists $\{h_k\}_{k \in \mathbb{N}} \subset D(\mathcal{L})$ such that $\bplim_{k \to \infty} h_k = 1$ and $\bplim_{k \to \infty} \mathcal{L}^n h_k = 0$. In particular, $\bplim_{k \to \infty} h_k = \mathbbm{1}_{E_0}$ in $C(E_0^\Delta)$. Inserting $h_k$ in \eqref{eq:auxEqn1} and letting $k \to \infty$, dominated convergence gives 
  \begin{equation*}
    \P(\tilde{Z}_{\tilde{\tau}_n \wedge t} \in E_0) = \lim_{k \to \infty}\E[h_k(\tilde{Z}_{\tilde{\tau}_n \wedge t})] = \lim_{k \to \infty}\E[h_k(\tilde{Z}_{0})] = \P(\tilde{Z}_0 \in E_0) = 1.
  \end{equation*} 
  Therefore, for any $n \in \mathbb{N}$, $\tilde{Z}$ is RCLL and $U_n$-valued on $[0,\tilde{\tau}_n]$. Thus, for any $n \in \mathbb{N}$, we may view~$\tilde{Z}^{\tilde{\tau}_n}$ as a $D_{E_0}[0,\infty)$-valued process and optional sampling applied to~\eqref{eq:martingaleProblemDef} (and the analogon of~\eqref{eq:auxEq14} for $\tilde{Z}$) shows that $\tilde{Z}^{\tilde{\tau}_n}$ is a solution to the stopped martingale problem for $(\mathcal{L}^n,\delta_{0} \otimes \mu_0,U_n)$. Thus, by (i), the laws of $\tilde{Z}^{\tilde{\tau}_n}$ and $Z^{\tau_n}$ coincide.
  \smallskip

  To show~(iii), first note that $\tilde{\tau}_n \leq \tilde{\tau}_{n+1}$ for all $n \in \mathbb{N}$ and thus $\tau := \lim_{n \to \infty} \tilde{\tau}_n$ is well-defined. Since $\tilde{Z}$ has left-limits in $E_0^\Delta$, also $Y_t := \lim_{n \to \infty} \tilde{Z}_{\tilde{\tau}_n \wedge t}$ is well-defined in $E_0^\Delta$. Furthermore, $Y_t = \Delta$ if and only if $\tau \leq t$. Since $Z$ has sample paths in $D_{E_0}[0,\infty)$, it holds that $\tau_n \to \infty$, $\mathbb{P}$-a.s., and so (ii) implies 
  \begin{equation}\label{eq:auxEqn15}
    \E[h(Y_t)]  = \lim_{n \to \infty} \E[h(\tilde{Z}^{\tilde{\tau}_n}_{t})] = \lim_{n \to \infty}\E[h(Z^{\tau_n}_{t})] = \E[h(Z_t)]
  \end{equation}
  for any $t \geq 0$ and $h \in C_0(E_0)$. Taking $\{h_k\}_{k \in \mathbb{N}} \subset C_0(E_0)$ with $\bplim_{k \to \infty} h_k = 1$ (and thus $\bplim_{k \to \infty} h_k = \mathbbm{1}_{E_0}$ in $C(E_0^\Delta)$), inserting $h_k$ in \eqref{eq:auxEqn15} and letting $k \to \infty$, one obtains 
  \begin{equation*}
    \P(\tau > t) = \P(Y_t \in E_0) = \lim_{k \to \infty}\E[h_k(Y_t)] = \lim_{k \to \infty} \E[h_k(Z_t)] = \P(Z_t \in E_0) = 1
  \end{equation*}
  for all $t \geq 0$. Hence, $\P(\tau = \infty) = 1$, $\tilde{\tau}_n \to \infty$ a.s. and $\tilde{Z}$ does not explode, i.e it has sample paths in $D_{E_0}[0,\infty)$. In particular, for any choice of $0 \leq t_0 < \dots < t_m$, $m \in \mathbb{N}$, $h_0,\dots,h_m \in C_0(E_0)$, 
  \begin{equation*} 
    \E\left[\prod_{k=0}^m h_k(\tilde{Z}_{t_k})\right] = \lim_{n \to \infty}\E\left[\prod_{k=0}^m h_k(\tilde{Z}_{t_k}^{\tilde{\tau}_n})\right] \stackrel{(ii)}{=} \lim_{n \to \infty}\E\left[\prod_{k=0}^m h_k(Z_{t_k}^{\tau_n})\right] = \E\left[\prod_{k=0}^m h_k(Z_{t_k}^{\tau_n})\right].
  \end{equation*}
  Therefore, also the finite-dimensional marginals of~$\tilde{Z}$ and~$Z$ coincide.
  \smallskip
 
  Finally, since now well-posedness of the RCLL-martingale problem for $(\mathcal{L},\delta_{0} \otimes \mu_0)$ in the case $\mu_0 = \delta_{x_0}$ and $x_0 \in E$ is established, Lemma~\ref{lem:BProperties}~(v) and \cite[Thm.~2.1]{Bhatt1993} imply that the RCLL-martingale problem for $(\mathcal{L},\delta_{0} \otimes \mu_0)$ is well-posed also for arbitrary $\mu_0 \in \mathcal{P}(E)$. The exact same argument as above now shows that uniqueness holds even in the class of progressively measurable solutions.
\end{proof}

\subsection{From Uniqueness of the Martingale Problem to Uniqueness for the Fokker--Planck Equation}\label{subsec:WellPosednessFOkkerPlanck} 

Finally, we put together all results obtained in the previous sections. When $\sigma$ is continuous, Theorem~\ref{thm:KurtzResult} can be applied. When~$\sigma$ is not continuous, the following extension of Theorem~\ref{thm:KurtzResult} will be required:\footnote{See footnote~\ref{fn:pregenerator} for a discussion why (ii) implies that $\mathcal{L}_y^0$ is a pre-generator (in the terminology of \cite{Kurtz1998}).}

\begin{theorem}[{\cite[Thm.~2.7]{Kurtz1998}}]\label{thm:KurtzResult2}
  Let $E_0$ and $F$ be locally compact, separable metric spaces, $D(\mathcal{L}^0)\subset C_b(E_0)$ and $\mathcal{L}^0\colon D(\mathcal{L}^0) \to C_b(E_0 \times F)$ linear. Let $\eta\colon E_0\times\mathcal{B}(F) \to [0,1]$ be a transition kernel and define  
  \begin{equation}\label{eq:factorizedOperator} 
    \mathcal{L}_\eta f(x) := \int_F \mathcal{L}^0 f(x,y)\, \eta(x,\d y), \quad f \in D(\mathcal{L}^0).  
  \end{equation}
  For any $y \in F$, define the linear operator $\mathcal{L}^0_y$ with domain $D(\mathcal{L}^0)$ on $C_b(E_0)$ by $f \mapsto \mathcal{L}^0f(\cdot,y)$. 
  Let $\nu \in \mathcal{P}(E_0 \times F)$ and suppose that 
  \begin{itemize}
    \item[(i)] $D(\mathcal{L}^0)$ is an algebra and separates points,
    \item[(ii)] for any $x \in E_0$ and $y \in F$ there exists a solution to the RCLL-martingale problem for $(\mathcal{L}^0_y,\delta_x)$,
    \item[(iii)] $\mathcal{L}_\eta$ and $(\mathcal{L}_\eta,\nu)$ satisfy the conditions (ii) and (iv) of Theorem~\ref{thm:KurtzResult}.
  \end{itemize}
  Then the conclusion of Theorem~\ref{thm:KurtzResult} is valid, i.e. uniqueness holds for the forward equation for $(\mathcal{L}_\eta,\nu)$.
\end{theorem}

In the next lemma we show how to obtain uniqueness for the Fokker--Planck equation for $(\mathcal{D},\mathcal{A})$ from uniqueness for $(D(\mathcal{L}),\mathcal{L})$.

\begin{lemma}\label{lem:assSatisfied}
  Suppose $\sigma$ and $(\mathcal{D},\mathcal{A})$ are as in Theorem~\ref{thm:KolmogorovUniqueness} and $(D(\mathcal{L}),\mathcal{L})$ is defined as in~\eqref{eq:InhomogeneousGenerator}.
  Let $\mu_0 \in \mathcal{P}(E)$, $\thor$ as in Assumption~\ref{ass:sigma} and define $\nu:= \delta_0 \otimes \mu_0$. Then the following statements hold:
  \begin{itemize}
    \item[(i)] If $\sigma \mathcal{A} f \in C_0([0,\thor]\times E)$ for all $f \in \mathcal{D}$, then the assumptions of Theorem~\ref{thm:KurtzResult} are satisfied.
    \item[(ii)] Suppose there exists $R > 0$ such that $|\sigma(t,x)|\leq R$ for all $(t,x) \in [0,\infty) \times E$. Define $E_0:=[0,\infty) \times E$, $F := [0,R]$, $D(\mathcal{L}^0):= D(\mathcal{L})$ and for $f \in \mathcal{D}$, $\gamma \in C_c^1[0,\infty)$, set
    \begin{equation*}
      \mathcal{L}^0 (f\gamma)((t,x),v) := \gamma(t)v\mathcal{A}f(x) + f(x) \gamma'(t) ,  \quad (t,x) \in E_0 ,\ v \in F,
    \end{equation*} 
    and linearly extend this definition of $\mathcal{L}^0$ to $D(\mathcal{L}^0)$. Finally, define $\eta((t,x),\cdot) := \delta_{\sigma(t,x)}(\cdot)$ and $\mathcal{L}_\eta$ as in \eqref{eq:factorizedOperator}. Then the assumptions of Theorem~\ref{thm:KurtzResult2} are satisfied and $(D(\mathcal{L}),\mathcal{L})$ coincides with $(D(\mathcal{L}^0),\mathcal{L}_\eta)$.
  \end{itemize}
  In particular, in both cases the conclusion of Theorem~\ref{thm:KurtzResult} is valid, i.e. uniqueness holds for the forward equation~\eqref{eq:timeInhomogeneousPIDE} for~$(\mathcal{L},\nu)$.
\end{lemma}

\begin{proof}
  To prove~(i), firstly note that by Assumption~\ref{ass:FellerProcess}~(i) $\mathcal{D}$ is an algebra and dense. Hence, $D(\mathcal{L})$ is an algebra and separates points. Secondly, by Lemma~\ref{lem:BProperties}~(v) the condition~\eqref{eq:separabilityOfDomain} is indeed satisfied. Thirdly, by Proposition~\ref{prop:ExistenceUnbounded} existence holds and fourthly, by Proposition~\ref{prop:UniquenessUnbounded} uniqueness holds. Therefore, assumptions (i)-(iv) of Theorem~\ref{thm:KurtzResult} are indeed satisfied.
  \smallskip
   
  To prove~(ii), notice that $\sigma$ is a measurable function, and thus $\eta$ is indeed a transition kernel. Furthermore, by definition we have
  \begin{equation*}
    \mathcal{L}h(t,x) = \int_0^R \mathcal{L}^0h((t,x),v)\, \eta((t,x),\d v) = \mathcal{L}_\eta h(t,x). 
  \end{equation*}
  Hence, $(D(\mathcal{L}),\mathcal{L})$ and $(D(\mathcal{L}^0),\mathcal{L}_\eta)$ indeed coincide and it only remains to show that the assumptions of Theorem~\ref{thm:KurtzResult2} are satisfied. Firstly, $D(\mathcal{L}^0)=D(\mathcal{L})$ is an algebra and separates points as argued in (i). Secondly, for any $v \in [0,R]$ and $(s_0,x) \in E_0$ a solution to the RCLL-martingale problem for $(\mathcal{L}_v^0,\delta_{(s_0,x)})$ can be constructed as follows: Let $M$ be a solution to the RCLL-martingale problem for $(\mathcal{A},\delta_x)$ and set $X_t := M_{v t}$. Then by elementary change of variable, 
  \begin{equation*} 
    f(X_t) - \int_0^t v \mathcal{A} f(X_s) \dd s  = f(M_{t v}) - \int_0^{t v} \mathcal{A} f(M_s) \dd s
  \end{equation*}
  for all $f \in \mathcal{D}, t \geq 0$, and thus $X$ is a solution to the (time-inhomogeneous) RCLL-martingale problem for $(v \mathcal{A},\delta_{x})$. Therefore, by Lemma~\ref{lem:BProperties}, $(t+s_0,X_t)_{t \geq 0}$ is a solution to the RCLL-martingale problem for $(\mathcal{L}_v^0,\delta_{(s_0,x)})$. Finally, since $\sigma$ is bounded, $(D(\mathcal{L}),\mathcal{L})$ satisfies \eqref{eq:separabilityOfDomain} by Lemma~\ref{lem:BProperties}~(v) and the martingale problem for $(\mathcal{L},\nu)$ is well-posed by Proposition~\ref{prop:UniquenessUnbounded}.
\end{proof}

After these preparations, we are ready to prove the main result in this section, Theorem~\ref{thm:KolmogorovUniqueness}~(ii).

\begin{proof}[Proof of Theorem~\ref{thm:KolmogorovUniqueness}~(ii)]
  Suppose $(p(t,\cdot))_{t \in [0,\thor]}$ and $(\tilde{p}(t,\cdot))_{t \in [0,\thor]}$ both satisfy~ \eqref{eq:measurabilityKolmogorovEqn} and \eqref{eq:KolmogorovEqn}. Defining for any $t \geq 0$ the measures $\nu_t:=\delta_t \otimes p(t,\cdot)$ and $\tilde{\nu}_t:=\delta_t \otimes \tilde{p}(t,\cdot)$, by Lemma~\ref{lem:BProperties}~(iv), $(\nu_t)_{t\geq 0}$ and $(\tilde{\nu}_t)_{t\geq 0}$ both satisfy \eqref{eq:timeInhomogeneousmeasurabilityKolmogorovEqn} and  \eqref{eq:timeInhomogeneousPIDE}. Under our assumptions, Lemma~\ref{lem:assSatisfied} implies that uniqueness holds for \eqref{eq:timeInhomogeneousPIDE}, i.e. $\tilde{\nu}_t = \nu_t$ for all $t \geq 0$ or $\delta_t \otimes p(t,\cdot) = \delta_t \otimes q(t,\cdot)$ for all $t \geq 0$. In particular, $\int_{E} f(x)\, p(s,\d x) = \int_{E} f(x)\, \tilde{p}(s,\d x)$ for all $f \in C_0(E)$ and all $s \in [0,\thor]$ and thus the assertion follows.
\end{proof}

\appendix

\section{Auxiliary results for the construction of time-changes}\label{sec:appendix}

In order to ensure the existence of the time-change~$\tau$ as defined by the random differential equation~\eqref{eq:ode time-change}, we used the following lemma concerning so-called Carath\'eodory differential equations.

\begin{lemma}\label{lem:timeinverse}
  Let $\thor > 0$ and consider the Carath\'eodory differential equation 
  \begin{equation}\label{eq:inverseStopping}
    \mathcal{T}(t) = \int_0^t \gamma(\mathcal{T}(s),s)\dd s, 
  \end{equation}
  where $\gamma \colon [0,\thor] \times [0,\infty)  \to [0,\infty)$ and the integral is understood in the Lebesgue sense. For some $S \in (0,\thor]$ and $T>0$ suppose that $\gamma(r,\cdot)$ is measurable for each $r \in [0,S]$, $\gamma(0,\cdot)$ is integrable on $[0,T]$ and there exists an integrable function $f\colon [0,T]\to [0,\infty)$ such that
  \begin{equation*} 
    |\gamma(r,t) - \gamma(s,t) | \leq f(t) |r -s| \quad  \text{for all}\quad r,s\in [0,S],\, t \in [0,T].
  \end{equation*}
  Then there exists a unique absolutely continuous function $\mathcal{T}\colon I \to [0,S]$ satisfying \eqref{eq:inverseStopping} for some interval $I \subset [0,\infty)$, where either there exists $T_0 \in (0,T]$ such that we may take $I = [0,T_0]$ and we have $\mathcal{T}(T_0) = S$ or we may take $I = [0,T]$ and have $\mathcal{T}(t) < S$ for all $t \leq T$. 
\end{lemma}

\begin{proof}
  Since 
  \begin{equation*}
    |\gamma(r,t)| \leq |\gamma(0,t)| + |f(t)| S\quad \text{for all}\quad r \in [0,S] 
  \end{equation*}
  and the right-hand side is integrable on $[0,T]$, $\gamma$ satisfies the Carath\'eodory Conditions in \cite[Chap.~1]{Filippov1988} and thus \cite[Chap.~1, Thm.~1]{Filippov1988} guarantees the existence of a solution~$\mathcal{T}$ on an interval $[0,T_0]$ for some $T_0 > 0$. The solution~$\mathcal{T}$ can be extended either to the whole interval $[0,T]$ provided $\mathcal{T}(t)\leq S$ for all $t\in [0,T]$ or to the interval $[0,T_0]$ for some $T_0 \in (0,T]$ with $\mathcal{T}(T_0) = S$ (see e.g. \cite[Chap.~1, Thm.~4]{Filippov1988}). Uniqueness of the solution follows by \cite[Chap.~1, Thm.~2]{Filippov1988}.
\end{proof}

Next, we provide a condition that is useful in verifying regularity of $H$ (see Definition~\ref{def:regular}) needed for the existence of the time-change in Lemma~\ref{lem:timechange} and the uniqueness in Lemma~\ref{lem:uniquenessOfTimeChange} below.

\begin{proposition}\label{prop:regular}
  Let $\mathcal{D} \subset C_0(E)$ dense in $C_0(E)$ and $\mathcal{A}\colon \mathcal{D} \to C_0(E)$ be linear. Suppose $M$ is a solution on $(\Omega,\mathcal{F},\mathbb{P})$ to the RCLL-martingale problem for $(\mathcal{A},\mu_0)$, for some $\mu_0 \in \mathcal{P}(E)$. Denote by $P$ the law on $D_E[0,\infty)$ of $M$. Then any $H \in \mathcal{D}$ with $H \geq 0$ is regular for $P$.
\end{proposition}

\begin{proof}
  Define $\rho$ as in \eqref{eq:rho} and recall that, by Definition~\ref{def:regular}, \eqref{eq:regularInLemma1} and \eqref{eq:regularInLemma2} have to be verified. Set 
  \[
    \rho_0 := \inf \left\lbrace s \in [0,\infty)\,:\, H(M_s) = 0 \right\rbrace. 
  \]
  Since $H$ is continuous and $M$ is RCLL, $H(M_{\rho})=0$ on $\{\rho < \infty\}$ and $\rho_0 \leq \rho$, $\P$-a.s. In particular, $\rho_0 = \rho$ on $\{\rho_0 = \infty\}$, and if $\{\rho_0 < \infty\}$ is a $\P$-null set, this already establishes the claim. Otherwise the probability measure $\tilde{\P}(\,\cdot\,):=\P(\,\cdot\, |\{\rho_0 < \infty\})$ is well-defined, $\rho_0 < \infty$, $\tilde{\P}$-a.s. and to prove the proposition we only need to show $\tilde{\P}(\rho_0 \geq \rho)=1$. To do so, on $\{\rho_0 < \infty\}$ define for any $t \geq 0$ the random time
  \[ 
    \tau(t):= \inf \left\lbrace s \in [0,\infty)\,:\,\int_0^s H(M_{u+\rho_0})^{-1} \dd u \geq t \right\rbrace .
  \] 
  Since $H(M_{\rho_0})=0$ and $\rho_0 \leq \rho$, $\tilde{\P}$-a.s., it suffices to establish that $\tilde{\P}$-a.s. for any $t \geq 0$, $\tau(t)=0$.

  \smallskip
  For the proof of the last statement one proceeds as follows: Since $H$ is bounded, $H(M_\rho)=0$ on $\{ \rho < \infty \}$, $\tilde{\P}$-a.s., and by footnote~\ref{fn:regular}, Lemma~\ref{lem:timechange} can be applied to the RCLL process $(M_{u+\rho_0})_{u \geq 0}$ on $(\Omega,\mathcal{F},\tilde{\P})$ with $\tilde{\sigma}=1$ and $\sigma=H$. This yields $\tilde{\P}$-a.s.,
  \begin{equation}\label{eq:auxEq61} 
    \tau(t)=\int_0^t H(M_{\tau(u)+\rho_0})\dd u, \quad t \geq 0,
  \end{equation}
  and $\tau(t) < \infty$ for any $t \geq 0$. Denote by $(\mathcal{F}_t)_{t \geq 0}$ the $\P$-usual augmentation of the filtration generated by $M$. Then $\rho$ and $\rho_0$ (possibly modified on a $\P$-null set, see \cite[Chap.~4, Cor.~3.13]{Ethier1986a}) are $(\mathcal{F}_t)_{t \geq 0}$-stopping times and thus
  \[
    \lbrace \tau(t)+\rho_0 \leq s \rbrace =  \lbrace \rho_0 \leq s \rbrace \cap \left( \left\lbrace \int_{\rho_0}^{s} H(M_{u})^{-1} \dd u  \geq t \right\rbrace \cup \{\rho \leq s - \rho_0 \} \right) \in \mathcal{F}_s
  \]
  shows that also $\tau(t)+\rho_0$ is a stopping time. By assumption on $H$, $M$ and $\mathcal{A}$ the process
  \begin{equation*}
    N_t := H(M_t) - H(M_0) - \int_0^t \mathcal{A} H (M_s) \dd s, \quad t \geq 0,
  \end{equation*}
  is an $(\mathcal{F}^M_t)_{t\geq 0}$-martingale and thus, by \cite[Lem.~II.67.10]{Rogers2000}, also an $(\mathcal{F}_t)_{t \geq 0}$-martingale. By the optional sampling theorem, for any $r \geq 0$, $\P$-a.s.,
  \[
    \E[N_{(\tau(t)+\rho_0)\wedge r} | \mathcal{F}_{\rho_0 \wedge r}] = N_{\rho_0 \wedge r} 
  \]
  or equivalently
  \[
    \E[H(M_{(\tau(t)+\rho_0)\wedge r})| \mathcal{F}_{\rho_0 \wedge r}] = H(M_{\rho_0 \wedge r}) + \E\left[\left.\int_{\rho_0 \wedge r}^{(\tau(t)+\rho_0)\wedge r} \mathcal{A} H(M_u) \dd u\right| \mathcal{F}_{\rho_0 \wedge r}\right]. 
  \]
  Multiplying by $\mathbbm{1}_{\lbrace \rho_0 \leq r \rbrace}$, using $\{\rho_0 \leq r\} \in \mathcal{F}_{\rho_0 \wedge r}$ and taking expectations gives 
  \[
    \E[H(M_{(\tau(t)+\rho_0)\wedge r})\mathbbm{1}_{\lbrace \rho_0 \leq r \rbrace}] = \E[H(M_{\rho_0 \wedge r})\mathbbm{1}_{\lbrace \rho_0 \leq r \rbrace}] + \E\left[\int_{\rho_0 \wedge r}^{(\tau(t)+\rho_0)\wedge r} \mathcal{A} H(M_u) \dd u \mathbbm{1}_{\lbrace \rho_0 \leq r \rbrace}\right].
  \]
  By assumption $\mathcal{D} \subset C_0(E)$ is dense and thus separating (see \cite[Chap.~3, Ex.~11]{Ethier1986a}) in the terminology of \cite[Chap.~3, Sec.~4]{Ethier1986a}. Therefore, by quasi-left continuity, \cite[Chap.~4, Thm.~3.12]{Ethier1986a}, 
  \[
    \lim_{r \to \infty} M_{(\tau(t)+\rho_0)\wedge r}\mathbbm{1}_{\lbrace \rho_0 \leq r \rbrace} = M_{\tau(t)+\rho_0}\mathbbm{1}_{\lbrace \rho_0 < \infty\rbrace} ,\quad \P\text{-a.s.},  
  \]
  and so using dominated convergence, boundedness and non-negativity of $H$,  $H(M_{\rho_0})=0$ on $\{\rho < \infty\}$ and setting $C:=\|\mathcal{A} H\|$, one estimates
  \begin{equation}\label{eq:Hestimate}\begin{aligned}
    \E[H(M_{\tau(t)+\rho_0})\mathbbm{1}_{\lbrace \rho_0 < \infty \rbrace}] & = \lim_{r \to \infty} \E[H(M_{(\tau(t)+\rho_0)\wedge r})\mathbbm{1}_{\lbrace \rho_0 \leq r \rbrace}] \\ 
    & = \lim_{r \to \infty} \left| \E\left[\int_{\rho_0}^{(\tau(t)+\rho_0)\wedge r} \mathcal{A} H(M_u) \dd u \mathbbm{1}_{\lbrace \rho_0 \leq r \rbrace}\right] \right| \\ &\leq C \E[\tau(t)\mathbbm{1}_{\lbrace \rho_0 < \infty \rbrace}].\end{aligned}
  \end{equation}
  Using \eqref{eq:auxEq61} and Tonelli's theorem for the first and \eqref{eq:Hestimate} for the second equality yields
  \begin{equation}\label{eq:auxEq60} 
    \E[\tau(t)\mathbbm{1}_{\lbrace \rho_0 < \infty \rbrace}]=\int_0^t \E[H(M_{\tau(u)+\rho_0})\mathbbm{1}_{\lbrace \rho_0 < \infty \rbrace}] \dd u \leq C \int_0^t \E[\tau(u)\mathbbm{1}_{\lbrace \rho_0 < \infty \rbrace}] \dd u 
  \end{equation}
  and so Gronwall's lemma implies that the left-hand side of \eqref{eq:auxEq60} is $0$ for any $t \geq 0$. But this implies that $\tilde{\P}$-a.s., $\tau(t) = 0$ for all $t \geq 0$ as desired.
\end{proof}

\subsection{Pathwise Uniqueness}

To verify that the random times $(\tau(t))_{t\in[0,\thor]}$ solving the differential equation~\eqref{eq:ode time-change} are indeed stopping times with respect to the filtration generated by the process $M$, we show pathwise uniqueness of the time-changed Markov process $X_t:=M_{\tau(t)}$ for $t\in [0,\thor]$.

\begin{lemma}\label{lem:uniquenessOfTimeChange}
  Let $\sigma$ and $M$ be given as in Lemma~\ref{lem:timechange}. $(\tau(t))_{t \in [0,\thor]}$ is the family of random times from Lemma~\ref{lem:timechange} with $\tau(t): = \tau(\thor)$ for $t > \thor$ and the time-changed process $X$ is given by $X _{t} := M_{\tau(t)}$ for $t \geq 0$. Suppose $M$ is $(\mathcal{F}_t)$-adapted. Then the following holds:
  \begin{itemize}
    \item[(i)]   The time-changed process $X$ has RCLL sample paths, $\P$-a.s.
    \item[(ii)]  Any RCLL process $\tilde X$ satisfying 
                 \begin{equation}\label{eq:timeChangeGeneral}
                   \tilde X_t = M_{\int_0^t \sigma(u,\tilde X_u) \dd u},\quad t\in[0,\infty),\, \P\text{-}a.s., 
                 \end{equation} 
                 is indistinguishable from $X$.
    \item[(iii)] The random times $(\tau(t))_{t \in [0,\thor]}$ are $(\mathcal{F}_t)$-stopping times.
  \end{itemize}
\end{lemma}

\begin{proof}
  \textit{(i)} Since $M$ has RCLL sample paths and $\tau$ is non-decreasing and absolutely continuous by Lemma~\ref{lem:timechange}, the time-changed process $X_{}$ has RCLL sample paths.\smallskip 
  
  \textit{(ii)} Let $\tilde X$ be an RCLL process satisfying equation~\eqref{eq:timeChangeGeneral}. Define the random time
  \begin{equation*}
    \tilde{\rho} := \thor \wedge \inf\big\{ t \geq 0 \,:\, H(\tilde X_t) = 0  \big\} 
  \end{equation*}
  and set
  \begin{equation*}
    \tilde{\tau}(s) := \int_0^s \sigma(u,\tilde X_u ) \dd u, \quad s\in [0,\infty). 
  \end{equation*}
  Notice that the integral is well-defined since $\sigma$ is bounded on compacts and $\tilde X$ is RCLL. Since $X_t=M_{\tau(t)}$ and $\tilde X_t= M_{\tilde{\tau}(t)}$, to verify that $X$ and $\tilde{X}$ are indistinguishable, it is sufficient to show that $\tau (t)= \tilde \tau(t)$ for every $t\in [0,\infty)$, $\P$-a.s. 
  
  By \cite[Lem.~3.31]{Leoni2009} $\tilde{\tau}$ is absolutely continuous with weak derivative $\tilde{\tau}'(u) = \sigma(u, \tilde X_u )$ for $u \in [0,\infty)$ and invertible on $[0,\tilde{\rho}\wedge t_0]$ by the definition of $\tilde{\rho}$. The inverse of $\tilde\tau$ is denoted by $\tilde{\mathcal{T}}$ with domain $[0,\tilde{\tau}(\tilde{\rho}\wedge t_0)]$. Because $\tilde{\mathcal{T}}$ is also strictly increasing and absolutely continuous, the chain rule (see \cite[Thm.~3.44]{Leoni2009}) gives
  \begin{equation*}
    1 = \frac{\dd}{\dd t} \tilde \tau (\tilde{\mathcal{T}}(t))= \sigma(\tilde{\mathcal{T}}(t), \tilde X_{\tilde{\mathcal{T}}(t)})\frac{\dd}{\dd t}\tilde{\mathcal{T}}(t)\quad \text{for almost all } t \in [0,\tilde{\tau}(\tilde{\rho}\wedge t_0)]. 
  \end{equation*}
  Combining this with fundamental theorem of calculus (see \cite[Thm. 3.30]{Leoni2009}), one has that $\tilde{\mathcal{T}}$ satisfies the integral equation
  \begin{equation*} 
    \tilde{\mathcal{T}}(t) = \int_0^t \sigma(\tilde{\mathcal{T}}(s), \tilde X_{\tilde{\mathcal{T}}(s)})^{-1} \dd s,\quad t \in [0,\tilde{\tau}(\tilde{\rho}\wedge t_0)].
  \end{equation*}
  Moreover, notice that $M_t= M_{\tilde \tau(\tilde{\mathcal{T}}(t))}=\tilde X_{\tilde{\mathcal{T}}(t)}$ for $t \in [0,\tilde{\tau}(\tilde{\rho}\wedge t_0)]$. Therefore, $\mathcal{T}(t) = \tilde{\mathcal{T}}(t)$ for $t \in [0,\tilde{\tau}(\tilde{\rho}\wedge t_0)\wedge \tau(t_0)]$ since the solution to this equation is unique on $[0,\tau(t_0)]$, see \eqref{eq:Delta}. Furthermore, we have $\tilde{\tau}(\tilde{\rho}\wedge t_0) \leq \rho \wedge \tilde \tau (t_0)$ since $t < \tilde{\tau}(\tilde\rho\wedge t_0)$ implies $\tilde{\mathcal{T}}(t) < \tilde{\rho}\wedge t_0$ and thus $t < \rho\wedge \tilde \tau (t_0)$, where we recall $\rho$ from \eqref{eq:rho} and that \eqref{eq:regularInLemma1}, \eqref{eq:regularInLemma2} holds. In conclusion, $\mathcal{T}(t) = \tilde{\mathcal{T}}(t)$ for $t \in [0,\tilde{\tau}(\tilde{\rho}\wedge t_0)]$, which leads to $\tilde{\tau}(s) = \tau(s)$ for $ s \in [0,\tilde{\rho}\wedge t_0]$.
  
  To see $\tilde{\tau}(s) = \tau(s)$ for $s>\tilde{\rho}\wedge t_0$, we first observe that $1/(H(M_s)\vee \varepsilon )$ is bounded for every $\varepsilon >0$ and $\tilde{\sigma}$ is bounded on compacts by Assumption~\ref{ass:sigma}. Applying a change of variables (\cite[Cor.~3.57]{Leoni2009}) and using monotone convergence gives 
  \begin{equation*}
    C t \geq \lim_{\varepsilon \to 0} \int_0^t \frac{\sigma(s,\tilde X_s)}{H(\tilde X_s) \vee \varepsilon} \dd s = \lim_{\varepsilon \to 0} \int_0^{\tilde{\tau}(t)} \frac{1}{H(M_s) \vee \varepsilon} \dd s = \int_0^{\tilde{\tau}(t)} \frac{1}{H(M_s)} \dd s,
  \end{equation*}
  which ensures $\tilde{\tau}(t) \leq \rho $ for all $t \geq 0$. Assuming $\tilde{\rho}< \thor$, there exist $\{t_n\}_{n \in \mathbb{N}} \subset [\tilde{\rho},\thor]$ with $t_n \downarrow \tilde{\rho}$ and $H(\tilde X_{t_n}) = 0$ and so $\rho \leq  \tilde{\tau}(\tilde{\rho})$ by \eqref{eq:timeChangeGeneral} and \eqref{eq:regularInLemma1}, \eqref{eq:regularInLemma2}. In this case $\tilde{\tau}(t) = \tilde{\tau}(\tilde{\rho}) = \rho =\tau(t) $ for all $t \geq \tilde{\rho}$. Assuming $\tilde{\rho} \geq \thor$, we have $\tilde{\tau}(t) = \tilde{\tau}(\thor)$ for all $t \geq \thor$ due to $\sigma(t,\cdot) = 0$ for $t > \thor$ and in particular $\tilde{\tau}(t) = \tilde{\tau}(\thor)=\tau (t)$ for $t \geq \thor$.\smallskip
  
  \textit{(iii)} In order to apply a result from \cite{Ethier1986a}, we consider the two-dimensional process $Y_t:=(t,M_t)$ and the time-changed process $(t, X_t)$ for $t\in [0,T]$. Hence, \cite[Chap.~6, Thm.~2.2~(b)]{Ethier1986a} implies that $\tau(t)$ is a stopping time with respect to the usual augmentation of the filtration generated by $M$, and thus also an $(\mathcal{F}_t)$-stopping time, where we keep in mind that the first component of $Y$ generates a trivial filtration.  
\end{proof}

\begin{corollary}\label{cor:MeasurabilityofTimeChange} 
  Let $\sigma$, $M$ and $\tilde X$ be given as in Lemma~\ref{lem:uniquenessOfTimeChange} and denote by $P$ the law (on $D_E[0,\infty)$) of $M$ under $\P$. Then the law of $\tilde X$ under $\P$ is uniquely determined by $P$ and $\sigma$.
\end{corollary}

\begin{proof} 
  By Lemma~\ref{lem:uniquenessOfTimeChange}~(ii), the law of~$\tilde X$ is identical to  the law of~$X$ under~$\P$. To show explicitly that the latter is uniquely determined by $P$ and $\sigma$, one proceeds as follows: Let $n \in \mathbb{N}$, $t_1,\ldots,t_n \in [0,\infty)$, $B_1,\ldots,B_n \in \mathcal{B}(E)$ and let $\pi_1\colon D_E[0,\infty)\times D_E[0,\infty) \to D_E[0,\infty)$ be the projection map on the first component. We have seen in the proof of Lemma~\ref{lem:uniquenessOfTimeChange} that there exist a unique solution to the time-change equation for $\P$-a.e. sample path $M(\omega)$. Hence, as in the proof of \cite[Chap.~6, Lem.~2.1]{Ethier1986a}, the map 
  \begin{equation*}
    \gamma \colon D_E[0,\infty)\times D_E[0,\infty) \to D_E[0,\infty), \quad \gamma(M,X):= M_{\int_0^{\cdot} \sigma (u,X_u)\dd u},
  \end{equation*}
  is Borel measurable and the set 
  \[
    C := \{ (m,x) \in D_E[0,\infty) \times D_E[0,\infty) \,:\, \gamma(m,x)=x, x_{t_1} \in B_1, \ldots, x_{t_n} \in B_n\} 
  \]
  is in $\mathcal{B}(D_E[0,\infty)^2)$. Then \cite[Appendix~11, Thm.~11.3]{Ethier1986a} implies that $\pi_1 C$ is in the $P$-completion of $\mathcal{B}(D_E[0,\infty))$ and thus
  \[
    \P(X_{t_1}\in B_1,\ldots,X_{t_n} \in B_n) = P(\pi_1 C) 
  \]
   is indeed uniquely determined by $P$ and $\sigma$. 
\end{proof}

\bibliography{quellenSkoro}
\bibliographystyle{amsalpha}

\end{document}